\documentclass[10pt,a4paper]{article}
\usepackage[margin=1.5in]{geometry}
\usepackage[utf8]{inputenc}
\usepackage{amsmath}
\usepackage{amsfonts}
\usepackage{amssymb}

\usepackage{amsthm}
\usepackage{tikz}
\usetikzlibrary{cd}
\usepackage{mathtools}
\usepackage{booktabs}

\usepackage{hyperref}

\theoremstyle{plain}
\newtheorem{thm}{Theorem}[section]
\newtheorem{cor}[thm]{Corollary}
\newtheorem{lem}[thm]{Lemma}
\newtheorem{prop}[thm]{Proposition}
\newtheorem{obs}[thm]{Observation}

\theoremstyle{definition}
\newtheorem{definition}[thm]{Definition}
\newtheorem{qu}[thm]{Question}
\newtheorem{conj}[thm]{Conjecture}

\newtheorem{rem}[thm]{Remark}
\newtheorem{ex}[thm]{Example}
\newtheorem{prob}[thm]{Problem}

\theoremstyle{plain}
\newtheorem{thmintro}{Theorem}

\newtheorem{defintro}[thmintro]{Definition}

\newcommand{\I}{\mathbb{I}}
\newcommand{\Q}{\mathbb{Q}}
\newcommand{\HH}{\mathbb{H}}

\newcommand{\R}{\mathbb{R}}
\newcommand{\C}{\mathbb{C}}
\newcommand{\Z}{\mathbb{Z}}

\newcommand{\N}{\mathbb{N}}
\newcommand{\CP}{\mathbb{CP}}

\newcommand{\fg}{\mathfrak{g}}

\newcommand{\fm}{\mathfrak{m}}

\newcommand{\fh}{\mathfrak{h}}

\newcommand{\cB}{\mathcal{B}}
\newcommand{\cC}{\mathcal{C}}
\newcommand{\cD}{\mathcal{D}}

\newcommand{\cH}{\mathcal{H}}
\newcommand{\cI}{\mathcal{I}}
\newcommand{\cJ}{\mathcal{J}}

\newcommand{\cR}{\mathcal{R}}

\newcommand{\cU}{\mathcal{U}}

\newcommand{\Oh}{\mathcal{O}}
\newcommand{\V}{\mathcal{V}}

\newcommand{\im}{\operatorname{im}}
\newcommand{\del}{\partial}
\newcommand{\delbar}{{\bar{\partial}}}
\newcommand{\Fbar}{\bar{F}}

\newcommand{\pr}{\operatorname{pr}}
\newcommand{\rk}{\operatorname{rk}}
\newcommand{\Id}{\operatorname{Id}}
\newcommand{\Pro}{\mathbb{P}}

\newcommand{\supp}{\operatorname{supp}}

\DeclareMathOperator{\gr}{gr}

\DeclareMathOperator{\mult}{mult}

\newcommand{\cHf}{\cH^{form}}
\newcommand{\cHDf}{\cH^{form}}

\newcommand{\CM}{\mathcal{CM}}
\newcommand{\DR}{\mathcal{DR}}
\newcommand{\HDR}{\mathcal{HDR}}
\newcommand{\HDRC}{\mathcal{HDRC}}
\newcommand{\RB}{\mathcal{RB}}
\newcommand{\FS}{\mathcal{FS}}

\newcommand{\img}[2][1]{\begin{gathered}\includegraphics[scale=#1]{#2}\end{gathered}}

\DeclareMathOperator{\Hom}{Hom}
\mathtoolsset{centercolon=true}
\newcommand{\Cdot}{{\raisebox{-0.7ex}[0pt][0pt]{\scalebox{2.0}{$\cdot$}}}}
\mathchardef\mhyphen="2D

\let\oldabstract\abstract
\let\oldendabstract\endabstract
\makeatletter
\renewenvironment{abstract}
{%
	{\list{}{\addtolength{\leftmargin}{0em} 
			\listparindent 0em%
			\itemindent    \listparindent%
			\rightmargin   \leftmargin%
			\parsep        \z@ \@plus\p@}%
		\item\relax}%
	{\endlist}%
	\oldabstract}
{\oldendabstract}
\makeatother

\usepackage{authblk}
\title{On linear combinations of cohomological invariants of compact complex manifolds}
\author{Jonas Stelzig} 
\date{}
\begin{document}

\maketitle
\begin{abstract}
We prove that there are no unexpected universal integral linear relations and congruences between Hodge, Betti and Chern numbers of compact complex manifolds and determine the linear combinations of such numbers which are bimeromorphic or topological invariants. This extends results in the Kähler case by Kotschick and Schreieder. We then develop a framework to tackle the more general questions taking into account `all' cohomological invariants (e.g. the dimensions of the higher pages of the Frölicher spectral sequence, Bott-Chern and Aeppli cohomology). This allows us to reduce the general questions to specific construction problems. We solve these problems in many cases. In particular, we obtain full answers to the general questions concerning universal relations and bimeromorphic invariants in low dimensions.
\end{abstract}
\section*{Introduction}

In \cite{kotschick_hodge_2013}, Kotschick and Schreieder solve the following three problems for compact K\"ahler manifolds:

\begin{enumerate}
\item Which linear relations between Hodge and Chern numbers are valid for all compact K\"ahler manifolds?
\item Which linear combinations of Hodge and Chern numbers are bimeromorphic invariants of all compact K\"ahler manifolds?
\item Which linear combinations of Hodge and Chern numbers are topological invariants of all compact K\"ahler manifolds?
\end{enumerate}

Roughly speaking, the answer for all three questions is: `Only the expected ones'. The third question was famously asked by Hirzebruch \cite{hirzebruch_problems_1954}. The version of this question concerning the Chern numbers only had previously been solved by Kotschick \cite{kotschick_characteristic_2009}, \cite{kotschick_topologically_2012}, using the complex bordism ring. A key new ingredient in \cite{kotschick_hodge_2013} which allowed the systematic treatment of Hodge numbers was the consideration of the Hodge ring of K\"ahler manifolds, that is, the graded ring of $\Z$-linear combinations of compact complex K\"ahler manifolds, where two manifolds are identified if they have the same Hodge numbers. Somewhat surprisingly, it is finitely generated.\medskip

The method of \cite{kotschick_hodge_2013} was recently applied to treat linear relations between Hodge numbers of varieties in positive characteristic in \cite{de_bruyn_hodge_2020}. By a different method, it was  shown that there are no unexpected polynomial relations between the Hodge numbers of compact K\"ahler manifolds \cite{paulsen_construction_2020} or varieties in positive characteristic \cite{de_bruyn_construction_2020}.\medskip

In complex dimension $\geq 3$, there are many open questions concerning the relation of general compact complex manifolds to K\"ahler manifolds on the one side and merely almost complex manifolds on the other side. For example, there is no known topological obstruction distinguishing complex from almost complex manifolds. On the other hand, there exist manifolds admitting both K\"ahler and non-K\"ahler structures, but there is no known example of a manifold for which every complex structure is K\"ahler (in particular this is unknown even on $\CP^n$, $n\geq 3$). In complex dimension $2$, the situation is better understood, and Hodge and Chern numbers play a key role in distinguishing between almost complex, complex and K\"ahler manifolds.\medskip

A systematic study of the spread of cohomological invariants for general complex manifolds therefore appears desirable. In this article, we take a first step in this direction by tackling questions $1$ -- $3$ for arbitrary compact complex manifolds. We do so both in a narrower and a wider sense:\medskip

First, in sections \ref{sec: Hodge, Betti, Chern}, \ref{sec: bimero} and \ref{sec: rat. Hirzebruch}, we answer the immediate analogues of the above questions when asked for all compact complex manifolds instead of just K\"ahler ones. The main results are the following: 
\begin{thmintro}\label{thmintro: HDRC rels}$\,$
\begin{enumerate}
\item All universal $\Z$-linear relations between the Hodge, Betti and Chern numbers of $n$-dimensional compact complex manifolds are combinations of the following:
\begin{enumerate}
\item Serre duality: $h^{p,q}=h^{n-p,n-q}$
\item Poincar\'e duality: $b_k=b_{2n-k}$
\item Degeneration of the Fr\"olicher spectral sequence: 
If $n\leq 2$, $b_k=\sum_{p+q=k} h^{p,q}$
\item Hirzebruch-Riemann-Roch: $\chi_p=td_p$
\item Euler characteristics: $\sum_{p}(-1)^p\chi_p=\chi$
\item Connected components: $h^{0,0}=b_0$
\end{enumerate}
\item The only universal congruences between the Hodge, Betti and Chern numbers of $n$-dimensional compact complex manifolds are combinations of the ones given in the first point, $b_n(X)\equiv 0 ~ (2)$ in case $n$ odd, and the congruences involving Chern numbers alone.
\end{enumerate}
\end{thmintro}

We note that the congruences between the Chern numbers alone are known \cite{atiyah-hirzebruch_klassen_1961}, \cite{stong_relations_1965}, \cite{hattori_relations_1966}. They are the integrality conditions imposed by the index theorem applied to all bundles associated with the tangent bundle.

\begin{thmintro}\label{thmintro: bimeromorphic invs}
A $\Z$-linear combination (resp. congruence) of Hodge and Chern numbers is a bimeromorphic invariant for all compact complex manifolds if and only if it is a linear combination (resp. congruence) of the $h^{p,0}$ and $h^{0,p}$ only.
\end{thmintro}

\begin{thmintro}\label{thmintro: top inv}For compact complex manifolds of dimension $\geq 3$, a rational linear combination of Hodge and Chern numbers is
\begin{enumerate}
\item a universal homeomorphism invariant, if and only if it is a universal diffeomorphism invariant, if and only it is a linear combination of the Betti numbers, if and only if it is a linear combination of the Euler characteristic and the number of connected components $h^{0,0}$.

\item a universal oriented homeomorphism invariant, if and only if it is a universal oriented diffeomorphism invariant, if and only it is a linear combination of the Betti numbers and the Pontryagin numbers, if and only if it is a linear combination of the Euler characteristic, the number of connected components $h^{0,0}$ and the Pontryagin numbers.
\end{enumerate}
\end{thmintro}

In all of these statements, one has to read `linear combination' as `linear combination modulo the relations listed in Theorem \ref{thmintro: HDRC rels}'. The restriction to $n\geq 3$ in Theorem \ref{thmintro: top inv} is for simplicity only; the situation in low dimensions is well understood.\medskip

The overall strategy of the proofs of these results parallels that of the K\"ahler case treated in \cite{kotschick_hodge_2013}. A central algebraic tool is the Hodge de Rham ring of compact complex manifolds $\HDR_\ast$, defined as formal $\Z$-linear combinations of compact complex manifolds where two such linear combinations are identified if they have the same Hodge and Betti numbers. The idea of considering these kinds of rings comes from \cite{kotschick_hodge_2013}. The analogous ring in the positive characteristic setting has been considered in \cite{de_bruyn_hodge_2020} and for the formal part of the arguments we can make use of calculations from that article. Like the Hodge ring of K\"ahler manifolds, $\HDR_\ast$ is finitely generated. However, because the Fr\"olicher spectral sequence degenerates at the first page for surfaces but in general does not in dimension $\geq 3$, there are more generators to consider than in the Kähler or positive characteristic case and this leads to the problem of constructing manifolds with prescribed cohomological properties in dimensions up to $4$, which is in part done using SageMath.\medskip

The second part of the article starts from the observation that for general compact complex manifolds, Hodge and Betti numbers are just the tip of an iceberg of cohomology theories that are not visible on compact K\"ahler manifolds (where they are determined by Dolbeault cohomology). The most classical ones are the higher pages of the Fr\"olicher spectral sequence \cite{frolicher_relations_1955}, Bott Chern cohomology \cite{bott_hermitian_1965} and Aeppli cohomology \cite{aeppli_exact_1962}. More recent examples include the Varouchas groups \cite{varouchas_proprietes_1986}, the cohomologies of the Schweitzer complex \cite{schweitzer_autour_2007}, the refined Betti numbers \cite{stelzig_structure_2018} and the `higher-page analogues' of Bott-Chern, Aeppli and Varouchas cohomologies \cite{popovici_higher-page_2020}.\medskip

In \cite{stelzig_structure_2018}, an abstract definition of cohomological functors was given, encompassing all the previously mentioned examples. It appears natural to pose questions $1$ -- $3$ taking into account all cohomological functors at once. One is led to consider the following ring:
\begin{defintro}\label{defintro: univ ring}
	The \textbf{universal ring of cohomological invariants} $\cU_\ast$ is the ring of formal $\Z$-linear combinations of biholomorphism classes of compact connected complex manifolds modulo the relation
	\[
	X\sim Y \quad:\Leftrightarrow\quad H(X)\cong H(Y)\text{ for all cohomological functors }H.
	\]
\end{defintro} 
In sections \ref{sec: refined de Rham ring}, \ref{sec: Frolicher numbers} we discuss the algebraic structure of the two natural quotients $\FS_\ast$ and $\RB_\ast$ of $\cU_\ast$ arising by considering, instead of all cohomological invariants, only the higher pages of the Fr\"olicher spectral sequence or only the refined Betti numbers (which are recalled in Section \ref{sec: refined Betti numbers}). Unlike most other cohomological invariants, both of these satisfy a K\"unneth formula, making $\FS_\ast$ and $\RB_\ast$ amenable to a more elementary treatment by embedding them into polynomial rings. On the other hand, using the results of \cite{stelzig_structure_2018}, one obtains an injection
 \[
 \cU_\ast\longrightarrow \FS_\ast\times\RB_\ast,
 \]
but the equations describing the image are rather complicated. Instead, in Section \ref{sec: univ ring} we give a diagrammatic description of elements in $\cU_\ast$ which allows to write down an algebraic candidate for $\cU_\ast$, defined by a few simple diagrammatic conditions. Unlike $\HDR_\ast$, the ring $\cU_\ast$ is not finitely generated.\medskip

In order to show that the algebraic candidate is actually $\cU_\ast$, one has to realize a set of algebraic generators by linear combinations of compact complex manifolds. The arising (infinitely many) construction problems are substantially harder than the ones arising in the first part of the article. We realize all generators in small dimensions and certain infinite sequences of generators in Section \ref{sec: constructions}. As a consequence we obtain:
  \begin{thmintro}\label{corintro: no unexp. rels coh. invs. ref. betttis}
  	  	For $n\leq 3$, there are no unexpected universal $\Z$-linear relations or congruences between refined Betti numbers of $n$-dimensional compact complex manifolds or, more generally, between all cohomological invariants, except possibly $e_2^{0,1}=e_3^{0,1}$ on threefolds.
  \end{thmintro}
  
  \begin{thmintro}\label{corintro: bimeromorphic coh. invs}
  		Let $K$ be a field. For $n\leq 4$ (resp. $n\leq 5$), there are no unexpected $K$-linear combinations of cohomological invariants (resp. refined Betti numbers) that are bimeromorphic invariants of $n$-dimensional compact complex manifolds.
  \end{thmintro}
What the expected relations or linear combinations in each case are is recalled in the main body of the text. The dimension restrictions could be immediately improved by construction of the missing generators for $\RB_\ast$ and $\cU_\ast$. These, and some related open questions, are surveyed in Section \ref{sec: constr. probs}.\medskip

\textbf{Acknowledgements:} D. Angella supplied the examples in Section \ref{sec: Mn}. Furthermore, I am very grateful to D. Angella, R. Coelho, D. Kotschick, A. Latorre, A. Milivojevic, N. Pia, G. Placini, S. Schreieder and L. Ugarte for insightful conversations and remarks related to various parts of this project. Finally, I thank the anonymous referees for many useful comments that improved the presentation.\medskip

\textbf{Notations and conventions:}
All complex manifolds we consider will be assumed compact. Denote by $\CM_\ast$ the graded ring which in rank $n$ consists of formal $\Z$-linear combinations of isomorphism classes of compact complex manifolds of pure dimension $n$ modulo the relation $[X]+[Y]=[X\sqcup Y]$, and with multiplication given by $[X]\cdot[Y]:=[X\times Y]$. We will be interested in quotients of this ring defined by cohomological conditions. We sometimes omit the brackets and simply write $X$ for the isomorphism class $[X]$.

\section{Preliminaries: Nilmanifolds}\label{sec: nil}
Very roughly speaking, the main results of this paper are proved in a two-step procedure: First, an algebraic calculation reduces the general statement to a managable (e.g.: finite) list of construction problems (of the sort `find a manifold with particular cohomological properties') and then one solves these construction problems. In almost all cases, we use nilmanifolds with left invariant complex structures for the second step, which is why we recall some elements of that theory here. We refer for example to \cite{rollenske_dolbeault_2010} and the references therein for further detail.\medskip

Let $\fg$ be a finite dimensional nilpotent real Lie algebra. Associated with it, there is a simply connected nilpotent Lie group $N$ and we may identify $\fg$ with the left invariant vector fields on $N$. An almost complex structure on $\fg$ is a direct sum decomposition of the complexified dual $\fg_\C^\vee=\fg^{1,0}\oplus\fg^{0,1}$ such that $\overline{\fg^{1,0}}=\fg^{0,1}$. It is said to be integrable, or a complex structure, if the exterior differential
\begin{align*}
	d:\fg_\C^\vee&\longrightarrow\Lambda^2\fg_\C^\vee\\
	\omega&\longrightarrow((X,Y)\mapsto -\omega([X,Y]))
\end{align*}
is of type $(1,0)+(0,1)$ with respect to the associated decomposition $\Lambda^\Cdot \fg_\C^\vee=\bigoplus_{p,q}\Lambda_{\fg}^{p,q}$, where $\Lambda_{\fg}^{p,q}=\Lambda^p\fg^{1,0}\otimes\Lambda^q\fg^{0,1}$. Hence, for any complex structure, there is a finite dimensional double complex $\Lambda_\fg:=(\Lambda_\fg^{\Cdot,\Cdot},\del,\delbar)$ which can be identified with the subcomplex of left-invariant forms on $N$. Conversely, the space $\Lambda_{\fg}^{1,0}=\fg^{1,0}$ together with the map $d:\Lambda_{\fg}^{1,0}\to\Lambda_{\fg}^{2,0}\oplus\Lambda_{\fg}^{1,1}$ determine $\fg$ and the complex structure.\medskip

By work of Mal'cev \cite{malcev_class_1951}, the Lie group $N$ admits a (automatically cocompact) lattice $\Gamma$ if and only if $\fg$ admits a rational structure, i.e. if there is a rational nilpotent Lie algebra $\fg_\Q$ such that $\fg_\Q\otimes\R\cong \fg$. Given such $\Gamma$, the associated quotient $X=X(\fg,\Gamma)=\Gamma\backslash N$ is called a nilmanifold. Any complex structure on $\fg$ induces one on $X$ and we obtain an inclusion of double complexes $\Lambda_\fg\subseteq A_X$. As an immediate consequence of \cite[Prop. 3.16]{rollenske_dolbeault_2010} we get:
\begin{thm}\label{thm: left invariant Dolbeault}
	For any real nilpotent Lie algebra $\fg$ (admiting a rational structure) and a fixed complex structure, there exists a lattice $\Gamma$ in the associated nilpotent Lie group $N$ such that the natural map
	\[
	H_{\delbar}^{p,q}(\Lambda_\fg)\longrightarrow H_{\delbar}^{p,q}(\Gamma\backslash N)
	\]
	is an isomorphism for all $p,q\in\Z$. In particular, the induced map in total (i.e. de Rham) cohomology is an isomorphism.
\end{thm}
\begin{rem}
	The induced map in de Rham cohomology is actually known to be an isomorphism for \textbf{all} lattices. The same is conjectured to hold for Dolbeault cohomology. This is known to be true in many special cases, in particular for $\dim_{\R}\fg\leq 6$.
\end{rem}

\textbf{If in this article we speak of a nilmanifold associated with a given nilpotent Lie algebra with complex structure, we always take the lattice to be chosen as above.}\medskip

Admitting the terminology from Section \ref{sec: univ ring} and Lemma \ref{lem: same cohomological invariants}, let us also record the following stronger form of Theorem \ref{thm: left invariant Dolbeault}, which will only be needed in Section \ref{sec: constructions}.

\begin{cor}\label{cor: left invariant cohomoogy invariants}
	For any $\fg$, $\Gamma$, $N$ as above, the natural map
	\[
	H(\Lambda_\fg)\longrightarrow H(\Gamma\backslash N):=H(A_{\Gamma\backslash N})
	\]
	is an isomorphism for any cohomological functor in the sense of Definition \ref{def: cohomological functors}.
\end{cor}

In particular, for fixed $\fg$ and complex structure, computation of any cohomological invariant (e.g. Hodge or Betti numbers) of the corresponding complex nilmanifold is a matter of finite dimensional linear algebra and can therefore in principle be done by a computer. For example, the following numbers, which will be used later, were first found using SageMath code inspired by \cite{angella_sagemath_2017}.

\begin{lem}\label{lem: snN calc}
	Let $X$ and $Y$ be nilmanifolds associated with the strongly non nilpotent complex structures on $8$-dimensional Lie algebras of family $I$ in \cite{latorre_complex_2020}, corresponding to the parameter values $(\delta,\epsilon,\nu,a,b)=(1,0,0,1,0)$, resp. $(\delta,\epsilon,\nu,a,b)=(1,0,0,1,1)$. Then the Betti numbers are given as
	\[	(b_0(X),...,b_8(X))=(b_0(Y),...b_8(Y))=(1,5,8,11,14,11,8,5,1)
	\]
and the total Hodge numbers $h^k=\sum_{p+q=k}h^{p,q}$ are
\begin{align*}
(h^0(X),...,h^8(X))&=(1,5,11,15,16,15,11,5,1)\\
(h^0(Y),...,h^8(Y))&=(1,5,11,14,14,14,11,5,1).
\end{align*}
\end{lem}
Finally, we note that the tangent bundle of any nilmanifold is trivial as a smooth complex vector bundle, since $\fg$, considered as left invariant vector fields, gives a global frame stable under $J$. The following direct consequence for the Chern classes $c_i(N):=c_i(TN)$ will be of key importance later on:
\begin{obs}\label{obs: c_i=0}
	A nilmanifold $N$ with left invariant complex structure has vanishing Chern classes: $c_i(N)=0\in H_{dR}^{2i}(N)$ for all $i\geq 1$.
\end{obs}
\section{Hodge, Betti and Chern numbers}\label{sec: Hodge, Betti, Chern}
The goal of this section is to prove Theorem \ref{thmintro: HDRC rels}. The overall strategy of proof is the same as in \cite{kotschick_hodge_2013}. For the algebraic part of the argument, we can re-use some results written up by van Dobben de Bruyn in \cite{de_bruyn_hodge_2020} for the purpuse of studying analogous questions for varieties in positive characteristic, so we keep our notation close to that article.

\begin{definition}\label{def: Hodge ring} Consider $\Z[x,y,z]$ as a graded ring with $|x|=|y|=0$ and $|z|=1$. The 
	\textbf{Hodge ring} $\cH_\ast$ is the image of the	map of graded rings
\[
h:\CM_\ast\longrightarrow \Z[x,y,z]
\]
induced by sending a complex manifold to its Hodge polynomial (augmented by a dimension counting variable as in \cite{kotschick_hodge_2013}):
\[
[X]\longmapsto h(X):= \sum_{p,q=0}^{\dim_\C X}h^{p,q}(X)x^py^q z^{\dim_\C X}.
\]
\end{definition}
That $h$ is a map of rings is a consequence of the K\"unneth formula for Dolbeault cohomology (see e.g. \cite{griffiths_principles_1978}). The subring  $\cH_\ast^K$ generated by images of compact K\"ahler manifolds is the Hodge ring of K\"ahler manifolds studied in \cite{kotschick_hodge_2013}.
\begin{definition}\label{def. lin rel. Hodge}
With the above notation, a \textbf{universal linear relation (or congruence) between the Hodge numbers} of $n$-dimensional compact complex manifolds is a map of abelian groups $R$ from $(\Z[x,y,z])_n$ to $\Z$ (resp. $\Z/m\Z$), that vanishes on the image of $h$, i.e.  $R\circ h\equiv0$.
\end{definition}

\begin{definition}
Denote by $\cH^{form}_\ast:=\bigoplus_{n\geq 0}\cH^{form}_n$ the graded subring of $\Z[x,y,z]$ such that $\cH^{form}_n$ is the free abelian group of formal Hodge polynomials of degree $n$, i.e. those of the form
\[
\sum_{n\geq p,q\geq 0} h^{p,q}x^py^qz^n\qquad\text{s.t. } h^{p,q}=h^{n-p,n-q}.
\]
\end{definition}
While the polynomial description of $\cH_\ast$ and $\cHf_\ast$ is convenient to take the multiplicative structure into account, the reader may prefer to picture an element $P=(\sum_{p,q}h^{p,q}x^py^q)z^n\in\cHf_n$ of a fixed degree as a (formal) Hodge diamond, e.g. for $n=2$:
{\small
\[
\begin{array}{ccccc}
	&&h^{2,2}&&\\[3mm]
	&h^{2,1}&&h^{1,2}&\\[3mm]
	h^{2,0}&&h^{1,1}&&h^{0,2}\\[3mm]
	&h^{1,0}&&h^{0,1}&\\[3mm]
	&&h^{0,0}&&
\end{array}
\]
}
\begin{rem}\label{rem: Thm A for Hodge}
	There is an inclusion $\cH_\ast\subseteq \cH^{form}_\ast$. We will see below that it is an equality. This is equivalent to saying that all universal linear relations between the Hodge numbers are given by Serre duality, i.e. the part of Theorem \ref{thmintro: HDRC rels} involving the Hodge numbers only. In fact, any additional relation would cut out a smaller subspace than $\cH^{form}_\ast$ in which $\cH_\ast$ has to be contained.
	The same argument (with a bigger ring) will also be used to prove the full Theorem \ref{thmintro: HDRC rels}.
	\end{rem}

As for the algebraic structure of $\cHDf_\ast$, we have:

\begin{thm}[van Dobben de Bruyn]\label{thm: alg str Hodge ring}
The map 
\begin{align*}
\phi: \Z[A,B,C,D]&\longrightarrow\cHDf_\ast\\
A&\longmapsto(1+xy)z\\
B&\longmapsto(x+y)z\\
C&\longmapsto xyz^2\\
D&\longmapsto(x+xy^2)z^2
\end{align*}
is a surjection of graded rings, where $|A|=|B|=1$ and $|C|=|D|=2$. The kernel $I$ of $\phi$ is the principal ideal generated by
\[
G:=D^2-ABD+C(A^2+B^2-4C).
\]
\end{thm}

As in \cite{kotschick_hodge_2013}, we will from now on write $A,B,C,D$ also for their images under $\phi$.

\begin{thm}\label{thm: H=Hform}
There is an equality
\[
\cH_\ast=\cH^{form}_\ast.
\]
One may take as generators (the images of) $\CP^1$, an elliptic curve $E$, any K\"ahler surface $S_K$ of signature $\pm 1$ (e.g. $\CP^2$) and any non-K\"ahler surface $S_{NK}$.
\end{thm}
\begin{proof}
As noted in \cite[Cor. 3]{kotschick_hodge_2013}, $h(\CP^1)=A$, $h(E-\CP^1)=B$, $h(S_K)\equiv \pm C\mod \langle A^2, B^2, AB\rangle$ and $A,B,C$ generate the Hodge ring of K\"ahler manifolds, $\cH^K_\ast$, i.e. the subring $\cH_\ast^{form}$ generated by polynomials $(\sum h^{p,q} x^py^q)z^n$ satisfying the additional relation $h^{p,q}=h^{q,p}$. Since a non-K\"ahler surface satisfies $h^{0,1}(S_{NK})=h^{1,0}(S_{NK})+1$ and $h^{0,2}(S_{NK})=h^{2,0}(S_{NK})$ (see e.g. \cite{barth_compact_1984}), its Hodge polynomial is congruent to $-D$ modulo $A^2$, $B^2$, $AB$ and $C$. \end{proof}

\begin{rem}
For concreteness, one may pick $S_K$ to be $\CP^2$ and $S_{NK}=H$ to be a Hopf surface. Then $C=h(\CP^1\times\CP^1-\CP^2)$ and $D+2C=h(E\times \CP^1-H)$.
\end{rem}
\begin{rem}
Theorem \ref{thm: H=Hform} was already known to D. Kotschick and S. Schreieder (unpublished).
\end{rem}

Next, we consider the Betti numbers:

\begin{definition}\label{def: de Rham ring}
	Let $\Z[t,z]$ be the graded polynomial ring with $|t|=0$, $|z|=1$. The \textbf{de Rham ring (of compact complex manifolds)} is the image of the natural map of graded rings
	\[
	dR:\CM_\ast\longrightarrow \Z[t,z]
	\]
	induced sending a manifold to its (augmented) Poincar\'e polynomial
	\[
	[X]\longmapsto dR(X):=\sum_{k\geq 0}b_k(X)t^k z^{\dim_\C X}.
	\]
\end{definition} 
\begin{definition}
Let $\DR^{form}_n\subseteq \Z[t,z]$ be the group of formal Poincar\'e polynomials, i.e. those of the form
\[\sum_{k=0}^{2n} b_k t^kz^n\qquad\text{s.t. } b_k=b_{2n-k}\text{ and } 2|b_n \text{ if } n\text{ is odd.}\]
The direct sum $\DR^{form}_\ast=\bigoplus \DR^{form}_n$ is a graded subring of $\Z[t,z]$.
\end{definition}

Note that in a sum as above $b_n$ is even if and only if the Euler characteristic $\sum (-1)^k b_k$ is even. 

There is a natural map 
$s: \cH_\ast^{form}\rightarrow\DR_\ast^{form}$ defined by $x,y\mapsto t$, $z\mapsto z$. However, as remarked in \cite{de_bruyn_hodge_2020}, one has $s\circ h(X)=dR(X)$ if and only if the Fr\"olicher spectral sequence degenerates at the first page.
\begin{thm}[van Dobben de Bruyn]\label{thm: alg str deRh ring}
Let $\Z[A,B,C,D]$ be graded as in Theorem \ref{thm: alg str Hodge ring}. The map
\begin{align*}
\psi: \Z[A,B,C,D]&\longrightarrow\DR_\ast^{form}\\
A&\longmapsto(1+t^2)z\\
B&\longmapsto 2tz\\
C&\longmapsto t^2z^2\\
D&\longmapsto(t+t^3)z^2
\end{align*}
is a surjection of graded rings with $\psi=s\circ\phi$. The kernel of $\psi$ is the ideal
\[
J=(A^2C-D^2, AB-2D, B^2-4C, BD-2AC).
\]
\end{thm}
Since $s:\cHf_\ast\rightarrow \DR_\ast^{form}$ is surjective and $\cHf_\ast=\cH_\ast$ is generated by manifolds in complex dimension $\leq 2$ for which the Fr\"olicher spectral sequence degenerates, so that $s\circ h=dR$ on those generators, we get:
\begin{cor}
	There is an equality $\DR_\ast=\DR^{form}_\ast$. 
\end{cor}

Next, we will consider Hodge and Betti numbers simultaneously. 

\begin{definition}\label{def: Hodge de Rham ring}
The \textbf{Hodge de Rham ring} $\HDR_\ast$ is the image of the diagonal map:
\begin{align*}
hdR:\CM_\ast&\longrightarrow \cH_\ast\times \DR_\ast\\
[X]&\longmapsto (h(X),dR(X))
\end{align*}
\end{definition}
As in \cite{de_bruyn_hodge_2020}, consider the following homomorphisms:
\begin{align*}
\chi:\cH_\ast&\rightarrow \Z[z]&\chi:\DR_\ast&\rightarrow \Z[z]&h^{0,0}:\cH_\ast&\rightarrow \Z[z]&b_0:\cH_\ast&\rightarrow \Z[z]\\
x,y&\mapsto -1&t&\mapsto -1&x,y&\mapsto 0&t&\mapsto 0
\end{align*}

\begin{lem}[van Dobben de Bruyn]\label{lem: ker chi b0}
The kernel of $(b_0,\chi):\DR\longrightarrow \Z[z]\times\Z[z]$ is the ideal generated by the two elements
\begin{align*}
d&:=(t+2t^2+t^3)z^2,\\
e&:=(t^2+2t^3+t^4)z^3.
\end{align*}
\end{lem}

Further, define the Fr\"olicher defect
\begin{align*}
FD: \cH_\ast\times\DR_\ast&\longrightarrow\DR_\ast\\
(a,b)&\longmapsto s(a)-b
\end{align*}
and denote by $FD_{\leq 2}$ the composition of $FD$ and the truncation $\DR_\ast\rightarrow \DR_{\leq 2}$.\medskip

Then a natural candidate for $\HDR_\ast$ is the ring
\[
\HDR_\ast^{form}:=\left\{(a,b) \in \cH^{form}_\ast \times \DR^{form}_\ast \bigg| \substack{\chi(a)=\chi(b)\\h^{0,0}(a)=b_0(b)\\FD_{\leq 2}(a,b)=0}\right\}.
\]
\begin{thm}\label{thm. Hodge de Rham ring}
	There is an equality $\HDR_\ast=\HDR_\ast^{form}$. In fact, $\HDR_\ast^{form}$ may be generated by the images of $\CP^1$, $\CP^2$ and a finite number of complex nilmanifolds of complex dimension $\leq 4$.
\end{thm}
\begin{proof}
  The ring $\cH_\ast$ may be generated by $\CP^1$, $\CP^2$, an ellitpic curve and the Kodaira-Thurston manifold. The latter two are nilmanifolds. For all of these the Fr\"olicher spectral sequence degenerates, so given $(a,b)\in\HDR^{form}_\ast$, the element $(a,s(a))$ lies in $\HDR_\ast$ and we may assume that $a=0$. Then $b$ has to be concentrated in degrees $\geq 3$ and satisfy $\chi(b)=\chi(0)=0$ and $b_0(b)=h^{0,0}(b)=0$. By Lemma \ref{lem: ker chi b0} the kernel $K$ of $(b_0,\chi):\DR_\ast\rightarrow \Z[z]\times\Z[z]$ is the ideal generated by the two elements $d$ and $e$. Thus, $K\cap \HDR_{\geq 3}$ is generated by $e, \psi(A)d,\psi(B)d,\psi(C)d, \psi(D)d$. Now note that for any $X\in \CM_\ast$ the element $(0,FD(X))=(h(X),s(h(X)))-(h(X),dR(X))$ lies also in $\HDR_\ast$. Therefore, it suffices to find (linear combinations of) compact complex manifolds s.t. $FD(X)$ equals  $e$, $\psi(A) d$, $\psi(B) d$, $\psi(C) d$ and $\psi(D) d$. Because $\psi(B)d=2e$ and $\psi(A)e=\psi(D)d$, the last two of these expressions are superfluous. The others may be realized as follows:\medskip
   
    $e=FD(I_{ii})$ for $I_{ii}$ a type $ii$-deformation of the Iwasawa-manifold, which has the following Betti and (total) Hodge numbers $h^k=\sum_{p+q=k}h^{p,q}$ (see \cite[p. 96]{nakamura_complex_1975}):
    \[
    (b_0(I_{ii}),...,b_6(I_{ii}))=(1,4,8,10,8,4,1)\qquad (h^0(I_{ii}),...,h^6(I_{ii}))=(1,4,9,12,9,4,1).
    \]  
    
	$\psi(A)d+e=FD(I)$ for $I$ the Iwasawa manifold itself, which has the following Betti and Hodge numbers (see \cite[p. 96]{nakamura_complex_1975}):
	\[
	 (b_0(I),...,b_6(I))=(1,4,8,10,8,4,1)\qquad (h^0(I),...,h^6(I))=(1,5,11,14,11,5,1).
	\]
	Finally, $\psi(C)d=FD(X-Y)$ for $X$ and $Y$ as in Lemma \ref{lem: snN calc}.\end{proof}

\begin{rem}
If we care about rational generators only, one can omit the four-dimensional nilmanifolds, since $2\psi(C)d=\psi(B) e$.
\end{rem}

More generally, we can consider linear relations between Hodge, Betti and Chern numbers. Let $\Omega_\ast^U$ denote the complex bordism ring. 

\begin{rem}
	There is an embedding into a polynomial ring on countably many generators $\tau:\Omega_\ast^U\subseteq \Z[b_1,b_2,...]$, where $|b_i|=2i$ and for an almost complex manifold of real dimension $2n$,  $\tau(X)=\sum_{I} c_I(X) b_I$, where $I$ runs over all $n$-tuples $(i_1,...,i_n)\in \Z_{\geq 0}^n$ such that $\sum i_k\cdot k=n$, $c_I(X)$ denote the Chern numbers $c_I(X):=\langle c_1(TX)^{i_1}\cdot...\cdot c_n(TX)^{i_n}, [X]\rangle$ and $b_I=b_{1}^{i_1}\cdot...\cdot b_{n}^{i_n}$. By definition and since $\Omega^U_\ast$ may be generated by compact complex manifolds, the image of $\tau$ in degree $n$ is identified with the kernel of all universal linear relations and congruences between the Chern numbers of $n$-dimensional compact complex manifolds (in the sense of Definition \ref{def. lin rel. Hodge}). These are completely known. In fact, there are only congruences, namely precisely those given by the Index theorem applied to all bundles associated with the tangent bundle  \cite{atiyah-hirzebruch_klassen_1961}, \cite{stong_relations_1965}, \cite{hattori_relations_1966}.
\end{rem}
\begin{definition}\label{def: Hodge de Rham Chern ring}
	The \textbf{Hodge de Rham Chern ring} $\HDRC_\ast$ is the image of the diagonal map $\CM_\ast\longrightarrow \HDR_\ast\times\Omega_\ast^U.$ The \textbf{Hodge Chern ring} $\cH\cC_\ast$ is the image of the diagonal map $\CM_\ast\longrightarrow \cH_\ast\times\Omega_\ast^U$. 
\end{definition}  
The Todd genera define a homomorphism $td_\ast:\Omega_\ast^U\rightarrow\Z[x,z]$ sending a class $[X]\in\Omega_n$ to the polynomial $\sum_{p=0}^n td_p(X) x^pz^n$. Following \cite{kotschick_hodge_2013}, we interpret the Hirzebruch genus as the homomorphism
\begin{align*}
\chi_\ast:\cH_\ast&\longrightarrow \Z[x,z]\\
y&\longmapsto -1.
\end{align*}
Setting $x=1$, it specializes to the signature $\sigma:\cH_\ast\rightarrow \Z[z]$, and setting $x=-1$, it specializes to the Euler characteristic $\chi$. For any $X\in \cH_n$, write $\chi_\ast(X)=\sum_p \chi_p(X) x^p z^n$. By precomposing with the projection to the first factor, we may also consider $\chi_\ast$ as a map with source $\HDR_\ast$.
Since the Chern classes of all complex nilmanifolds vanish (Observation \ref{obs: c_i=0}), Hirzebruch-Riemann-Roch together with Theorem \ref{thm. Hodge de Rham ring} yields (c.f. \cite[Thm. 8]{kotschick_hodge_2013}):
\begin{lem}
	The image $\cH ir_\ast$ of $\chi_\ast$ is a polynomial ring, generated by the images of $h(\CP^1)$ and $h(\CP^2)$. The kernel of $\chi_\ast$ may be generated by complex nilmanifolds.
\end{lem}

\begin{thm}\label{thm: HDRC ring diagonal}
There is an equality  \[\HDRC_\ast=\left\{(a,b)\in\HDR_\ast \times \Omega_\ast^U \mid\chi_\ast(a)=td_\ast(b)\right\}.\]
\end{thm}
\begin{proof}
By Hirzebruch-Riemann-Roch, there is an inclusion from the left hand side to the right hand side. On the other hand, if $(hdR(X),[Y])\in\HDR_\ast\times \Omega_\ast^U$ s.t. $\chi_\ast (X)=td_\ast(Y)$ for all $p$, then $\chi_\ast(X)=\chi_\ast(Y)$ for all $p$. Hence, $hdR([Y]-[X])$ lies in the kernel of $\chi_\ast$. This kernel is generated by complex nilmanifolds, which by Observation \ref{obs: c_i=0} have vanishing Chern classes, so in particular represent the zero element in $\Omega_\ast^U$. Thus, writing $hdR(Y-X)=hdR(M)$, where $M$ is a linear combination of nilmanifolds, we see that
\begin{align*}
(hdR(X),[Y])&=(hdR(Y),[Y])-(hdR(Y-X), 0)\\
&=(hdR(Y),[Y])-(hdR(M),[M])
\end{align*}
lies in the image of the diagonal map.
\end{proof}
\begin{rem}\label{rem: HC diagonal}
	Analogously, $\cH\cC_\ast=\{(a,b)\in\cH_\ast \times \Omega_\ast^U \mid\chi_\ast(a)=td_\ast(b)\}$.
\end{rem}

Theorem \ref{thmintro: HDRC rels} from the introduction is now a reformulation of what we have shown in this section. In fact, as in Definition \ref{def. lin rel. Hodge} we may now say:
\begin{definition}\label{def.}
	A \textbf{universal linear relation (resp. congruence) between Hodge, Betti and Chern numbers} of $n$-dimensional compact complex manifolds is a map of abelian groups
	\[
	(\Z[t,z]\times\Z[x,y,z]\times\Z[b_1,b_2,...])_n\longrightarrow \Z
	\]
	(resp to $\Z/m\Z$) which vanishes identically on $\cH\DR\cC_n$.
\end{definition}

\begin{proof}[Proof of Theorem \ref{thmintro: HDRC rels}]
We may argue just as in Remark \ref{rem: Thm A for Hodge}:  In fact, by the defintions and Theorem \ref{thm: H=Hform}, Theorem \ref{thm. Hodge de Rham ring} and Theorem \ref{thm: HDRC ring diagonal}, $\cH\DR\cC_\ast$ is exactly the subspace given by all the relations and congruences described in Theorem \ref{thmintro: HDRC rels}. Hence, any other relation (resp. congruence) has to be a linear combination of these, since otherwise it would cut out an even smaller subspace.
\end{proof}

\section{Bimeromorphic invariants}\label{sec: bimero}

Recall (Definition \ref{def: Hodge ring}) that the Hodge ring $\cH_\ast\subseteq \Z[x,y,z]$ is the ring generated by the Hodge polynomials of compact complex manifolds. The differences of bimeromorphic compact complex manifolds generate an ideal $\cB\subseteq \cH_\ast$ and a linear combination (resp. congruence) of Hodge numbers of $n$-dimensional compact complex manifolds invariant under bimeromorphisms is nothing but a map $\cH_n\to \Z$ (resp. $\cH_n\to\Z/m\Z$ for some $m$) that factors over $\cH_n/\cB_n$. We therefore have to describe $\cB$. For this, it will be useful to consider the map
\[
p:\cH_\ast\longrightarrow \Z[x,y,z]/(xy)
\]
induced by the projection $\Z[x,y,z]\to\Z[x,y,z]/(xy)$. Similarly, consider the map $p^{\cH\cC}:=p\circ \pr$, where 
$\pr:\cH\cC_\ast\to\cH_\ast$ is the projection (c.f. Definition \ref{def: Hodge de Rham ring}) and denote by $\cB^{\cH\cC}\subseteq \cH\cC_\ast$ the ideal generated by differences of bimeromorphic compact complex manifolds.

\begin{thm}\label{thm: bimeromorphic invar. Hodge Chern}$\,$
	\begin{enumerate}
		\item The degree $n$ part of the image of $p$ (and a fortiori of $p^{\cH\cC}$) is free of rank $2n$, with a basis given by 
		\[z^n, xz^n, yz^n, x^2z^n, y^2z^n...,x^{n-1}z^n,y^{n-1}z^n, (x^n+y^n)z^n.\]
		\item There are equalities $\ker p = (C) = \cB$.
		\item There is an equality $\ker p^{\cH\cC}=\cB^{\cH\cC}$.
	\end{enumerate}
\end{thm}
\begin{proof}
	Parts $1$ and $2$ follow with exactly the same proof as in the positive characteristic setting $\cite{de_bruyn_hodge_2020}$, using that the numbers $h^{p,0}$ and $h^{0,q}$ are bimeromorphic invariants. For this reason, also the inclusion $\cB^{\cH\cC}\subseteq\ker p^{\cH\cC}$ in $3.$ is clear. For the other inclusion, we use the observation made in \cite[Thm. 13]{kotschick_hodge_2013} that there exists a basis sequence $\beta_1=\CP^1,\beta_2=\CP^2-\CP^1\times\CP^1,...$ for $\Omega^U_\ast$ with $\beta_i\in\cB^{\cH\cC}$ for $i\geq 2$. 
	
	Since $\cH\cC_\ast$ is generated by $E,H,\beta_1,\beta_2,...$, we get a composition of surjective maps
	\[
	\Z[E,H,\beta_1]\longrightarrow \cH\cC_\ast/\cB^{\cH\cC}\overset{\pr}{\longrightarrow}\cH_{\ast}/\cB=\cH_\ast/(C)=\Z[A,B,D]/(D^2-ABD).
	\]
Because in $\cH_\ast/(C)$ we have $A=\beta_1$, $B=E-\beta_1$ and $D=E\times\beta_1-H$, the kernel of this composition is the principal ideal generated by the element $(E\beta_1-H)^2-\beta_1(E-\beta_1)(E\beta_1-H)$, which maps to zero in $\cH\cC_\ast/\cB^{\cH\cC}$. Thus, the map $\cH\cC_\ast/\cB^{\cH\cC}\to\cH_\ast/\cB$ is an isomorphism. Since $\ker p^{\cH\cC}=\pr^{-1}(\ker p)=\ker \pr^{-1}(\cB)$, this completes the proof. \end{proof}

\begin{proof}[Proof of Theorem \ref{thmintro: bimeromorphic invs}] A linear combination (resp. congruence) as in the theorem can be considered as a map of abelian groups $\cH\cC_n/\cB^{\cH\cC}_n\to R$ for $R=\Z$ or $\Z/m\Z$. Since $\cB^{\cH\cC}=\ker p^{\cH\cC}$, this means it is given as a linear combination of the coefficients of the monomials in part $1$ of Theorem \ref{thm: bimeromorphic invar. Hodge Chern}, which are the Hodge numbers mentioned in Theorem \ref{thmintro: bimeromorphic invs}.
\end{proof}

\section{The rational Hirzebruch problem for general complex manifolds}\label{sec: rat. Hirzebruch}

The goal of this section is to prove Theorem \ref{thmintro: top inv} from the introduction. Consider the ideal $\cI\subseteq \HDR_\ast\otimes\Q$ in the (rationalized) Hodge de Rham ring from section \ref{sec: Hodge, Betti, Chern}, generated by differences of homeomorphic complex manifolds.

\begin{thm}\label{thm: ideal homeomorphic, Hodge}
In degrees $n\geq 3$, $\cI$ coincides with the kernel of the map
\[
\pr:\HDR_\ast\otimes\Q\longrightarrow \DR_{\ast}\otimes\Q.
\]
Under the projection $\HDR_\ast\to \cH_\ast$ it maps bijectively to the kernel of
\[
(h^{0,0},\chi):\cH_\ast\otimes\Q\longrightarrow \Q[z]\times\Q[z].
\]
\end{thm}

\begin{proof}
Since the Betti numbers are homeomorphism invariants, $\cI\subseteq K:=\ker\pr$. On the other hand, $K$ consists of all elements $(a,0)\in\HDR_\ast\subseteq\cH_\ast\times\DR_\ast$. They have to satisfy the relations $\chi(a)=\chi(0)=0$, $h^{0,0}(a)=b_0(0)=0$ and $FD_{\leq 2}(a,0)=s(a)_{\leq 2}=0$. I.e. we have an identification $K=K'\cap \ker s_{\leq 2}\subseteq\cH_\ast$ where $K':=\ker (h^{0,0},\chi)\subseteq\cH_\ast$. Since $FD_{\leq 2}(a,0)=0$ for all $a\in\cH_{\geq 3}$, this gives the identification $K=K'$ in degrees $\geq 3$. Since there is a commutative diagram
\[
\begin{tikzcd}
\cH_\ast\ar{r}{s}\ar{rd}[swap]{(h^{0,0},\chi)}&\DR_\ast\ar{d}{(b_0,\chi)}\\
&\Q[z]\times\Q[z]
\end{tikzcd}
\]
using Lemma $\ref{lem: ker chi b0}$, we obtain that $K'=s^{-1}\langle d,e\rangle=\langle \tilde{d},\tilde{e},\ker s\rangle$, where $\tilde{d}$, $\tilde{e}$ are arbitrary preimages of $d$ and $e$. We will choose $\tilde{d}=(x+2xy+xy^2)z^2$ and $\tilde{e}=(x^2+x^2y+xy^2+xy^3)z^3$. By Theorem \ref{thm: alg str deRh ring} and Theorem \ref{thm: alg str Hodge ring}, the kernel of $s$ is generated by $Q=A^2C-D^2$, $R:=BD-2AC$, $S:=B^2-4C$ and $T:=AB-2D$. Thus, $K'\cap\cH_{\geq 3}$ has the following generators:
\begin{align*}
A\tilde{d}, B\tilde{d}, \tilde{e}, AS, BS, AT, BT, R&&\text{in degree }3\\
C\tilde{d}, D\tilde{d}, CS, DS, CT, DT, Q&&\text{in degree }4
\end{align*}
We have to find  homeomorphic complex manifolds whose difference realizes these generators.\medskip

\textbf{Products with degree 2:} That $S$ can be realized as by a combination of differences of orientation reversingly homeomorphic complete intersections $X_1-Y_1$ and $X_2-Y_2$ has been shown in the proof of \cite[Thm. 10]{kotschick_hodge_2013}, as a consequence of the results in \cite{kotschick_orientation-reversing_1992}. By Theorem \ref{thm: H=Hform}, we may therefore realize $AS,BS,CS,DS$ as differences of homeomorphic manifolds.\medskip

In constrast, $T=(-x+y+x^2y-xy^2)z^2$ can never be realized. In fact, if two complex surfaces are homeomorphic, then their first Betti numbers have to coincide. But for surfaces, the Hodge numbers $h^{1,0}$ and $h^{0,1}$ (and therefore also $h^{2,1}$ and $h^{1,2}$) are determined by the first Betti number, so they would have to agree as well. Thus, we really have to realize all the other generators in degrees $3$ and $4$. We will do so in all cases by considering different left-invariant structures on a fixed nilmanifold:\medskip
 
\textbf{Degree 3:}  $\tilde{e}=h(I_{ii}-I_{iii})$ is realized by the difference  of a type $(ii)$ and a type $(iii)$ deformation of the Iwasawa manifold, see \cite[p. 96]{nakamura_complex_1975}.
\[\resizebox{2cm}{!}{$
\begin{array}{ccccccccc}
& & & &0& & & &\\[3mm]
& & &0& &0& & &\\[3mm]
& &0& &0& &1& &\\[3mm]
&0& &1& &1& &0&\\[3mm]
& &1& &0& &0& &\\[3mm]
& & &0& &0& & &\\[3mm]
& & & &0& & & &\\[3mm]
\end{array}$}
=
\resizebox{2cm}{!}{$
\begin{array}{ccccccccc}
& & & &1& & & &\\[3mm]
& & &2& &2& & &\\[3mm]
& &2& &5& &2& &\\[3mm]
&1& &5& &5& &1&\\[3mm]
& &2& &5& &2& &\\[3mm]
& & &2& &2& & &\\[3mm]
& & & &1& & & &\\[3mm]
\end{array}$}
-
\resizebox{2cm}{!}{$
\begin{array}{ccccccccc}
& & & &1& & & &\\[3mm]
& & &2& &2& & &\\[3mm]
& &2& &5& &1& &\\[3mm]
&1& &4& &4& &1&\\[3mm]
& &1& &5& &2& &\\[3mm]
& & &2& &2& & &\\[3mm]
& & & &1& & & &\\[3mm]
\end{array}$}
\]

$B\tilde{d}=h(X-Y)$ where the underlying manifold of both $X$ and $Y$ is a nilmanifold associated with the Lie algebra $\mathfrak{h}_{15}$ and the complex structures correspond to parameters $(\rho,B,c)=(1,2,0)$ and $(1,B,c)$ with $B\neq0\neq c$ in the notation of \cite{ceballos_invariant_2016}, see \cite[proof of Thm. 4.1]{ceballos_invariant_2016}.
\[\resizebox{2cm}{!}{$
\begin{array}{ccccccccc}
& & & &0& & & &\\[3mm]
& & &0& &0& & &\\[3mm]
& &0& &1& &1& &\\[3mm]
&0& &2& &2& &0&\\[3mm]
& &1& &1& &0& &\\[3mm]
& & &0& &0& & &\\[3mm]
& & & &0& & & &\\[3mm]
\end{array}$}
=
\resizebox{2cm}{!}{$
\begin{array}{ccccccccc}
& & & &1& & & &\\[3mm]
& & &2& &1& & &\\[3mm]
& &2& &4& &2& &\\[3mm]
&1& &5& &5& &1&\\[3mm]
& &2& &4& &2& &\\[3mm]
& & &1& &2& & &\\[3mm]
& & & &1& & & &\\[3mm]
\end{array}$}
-
\resizebox{2cm}{!}{$
\begin{array}{ccccccccc}
& & & &1& & & &\\[3mm]
& & &2& &1& & &\\[3mm]
& &2& &3& &1& &\\[3mm]
&1& &3& &3& &1&\\[3mm]
& &1& &3& &2& &\\[3mm]
& & &1& &2& & &\\[3mm]
& & & &1& & & &\\[3mm]
\end{array}$}
\]
  
$A\tilde{d}-B\tilde{d}+2\tilde{e}= h(I_{i}-I_{ii})$ is realized by the difference of a type $(i)$ and a type $(ii)$ deformation of the Iwasawa manifold, hence $A\tilde{d}\in \cI$, see again \cite{nakamura_complex_1975}.
\[
\resizebox{2cm}{!}{$
	\begin{array}{ccccccccc}
	& & & &0& & & &\\[3mm]
	& & &0& &1& & &\\[3mm]
	& &0& &2& &0& &\\[3mm]
	&0& &1& &1& &0&\\[3mm]
	& &0& &2& &0& &\\[3mm]
	& & &1& &0& & &\\[3mm]
	& & & &0& & & &\\[3mm]
	\end{array}$}
-
\resizebox{2cm}{!}{$
	\begin{array}{ccccccccc}
	& & & &0& & & &\\[3mm]
	& & &0& &0& & &\\[3mm]
	& &0& &1& &1& &\\[3mm]
	&0& &2& &2& &0&\\[3mm]
	& &1& &1& &0& &\\[3mm]
	& & &0& &0& & &\\[3mm]
	& & & &0& & & &\\[3mm]
	\end{array}$}
+ 2 \cdot
\resizebox{2cm}{!}{$
	\begin{array}{ccccccccc}
	& & & &0& & & &\\[3mm]
	& & &0& &0& & &\\[3mm]
	& &0& &0& &1& &\\[3mm]
	&0& &1& &1& &0&\\[3mm]
	& &1& &0& &0& &\\[3mm]
	& & &0& &0& & &\\[3mm]
	& & & &0& & & &\\[3mm]
	\end{array}$}
=
\resizebox{2cm}{!}{$
\begin{array}{ccccccccc}
& & & &1& & & &\\[3mm]
& & &2& &3& & &\\[3mm]
& &2& &6& &3& &\\[3mm]
&1& &6& &6& &1&\\[3mm]
& &3& &6& &2& &\\[3mm]
& & &3& &2& & &\\[3mm]
& & & &1& & & &\\[3mm]
\end{array}$}
-
\resizebox{2cm}{!}{$
	\begin{array}{ccccccccc}
	& & & &1& & & &\\[3mm]
	& & &2& &2& & &\\[3mm]
	& &2& &5& &2& &\\[3mm]
	&1& &5& &5& &1&\\[3mm]
	& &2& &5& &2& &\\[3mm]
	& & &2& &2& & &\\[3mm]
	& & & &1& & & &\\[3mm]
	\end{array}$}
\]

$BT=AS+2B\tilde{d}-4\tilde{e}$, so it is redundant.
\[\resizebox{2cm}{!}{$
	\begin{array}{ccccccccc}
	& & & &0& & & &\\[3mm]
	& & &0& &0& & &\\[3mm]
	& &1& &0& &-1& &\\[3mm]
	&0& &0& &0& &0&\\[3mm]
	& &-1& &0& &1& &\\[3mm]
	& & &0& &0& & &\\[3mm]
	& & & &0& & & &\\[3mm]
	\end{array}$}=
\resizebox{2cm}{!}{$
	\begin{array}{ccccccccc}
	& & & &0& & & &\\[3mm]
	& & &0& &0& & &\\[3mm]
	& &1& &-2& &1& &\\[3mm]
	&0& &0& &0& &0&\\[3mm]
	& &1& &-2& &1& &\\[3mm]
	& & &0& &0& & &\\[3mm]
	& & & &0& & & &\\[3mm]
	\end{array}$}
+2\cdot 
\resizebox{2cm}{!}{$
	\begin{array}{ccccccccc}
	& & & &0& & & &\\[3mm]
	& & &0& &0& & &\\[3mm]
	& &0& &1& &1& &\\[3mm]
	&0& &2& &2& &0&\\[3mm]
	& &1& &1& &0& &\\[3mm]
	& & &0& &0& & &\\[3mm]
	& & & &0& & & &\\[3mm]
	\end{array}$}
-4\cdot 
\resizebox{2cm}{!}{$
	\begin{array}{ccccccccc}
	& & & &0& & & &\\[3mm]
	& & &0& &0& & &\\[3mm]
	& &0& &0& &1& &\\[3mm]
	&0& &1& &1& &0&\\[3mm]
	& &1& &0& &0& &\\[3mm]
	& & &0& &0& & &\\[3mm]
	& & & &0& & & &\\[3mm]
	\end{array}$}
\]

$AT+A\tilde{d}-2B\tilde{d}+3\tilde{e}+BT=h(X-Y)$ where, in the notation of \cite{ceballos_invariant_2016}, $X$ and $Y$ are nilmanifolds associated with $\mathfrak{h}_{15}$ and the complex structures correspond to parameter values $(\rho,B,c)=(0,1,c)$ with $c\neq 1$ and $(\rho,B,c)=(1,B,c)$ with $B\neq 0\neq c$. Hence, $AT\in\cI$.
\begin{align*}
\resizebox{2cm}{!}{$
	\begin{array}{ccccccccc}
	& & & &0& & & &\\[3mm]
	& & &1& &-1& & &\\[3mm]
	& &0& &0& &0& &\\[3mm]
	&0& &0& &0& &0&\\[3mm]
	& &0& &0& &0& &\\[3mm]
	& & &-1& &1& & &\\[3mm]
	& & & &0& & & &\\[3mm]
	\end{array}$}
&+
\resizebox{2cm}{!}{$
	\begin{array}{ccccccccc}
	& & & &0& & & &\\[3mm]
	& & &0& &1& & &\\[3mm]
	& &0& &2& &0& &\\[3mm]
	&0& &1& &1& &0&\\[3mm]
	& &0& &2& &0& &\\[3mm]
	& & &1& &0& & &\\[3mm]
	& & & &0& & & &\\[3mm]
	\end{array}$}
-2\cdot
\resizebox{2cm}{!}{$
	\begin{array}{ccccccccc}
	& & & &0& & & &\\[3mm]
	& & &0& &0& & &\\[3mm]
	& &0& &1& &1& &\\[3mm]
	&0& &2& &2& &0&\\[3mm]
	& &1& &1& &0& &\\[3mm]
	& & &0& &0& & &\\[3mm]
	& & & &0& & & &\\[3mm]
	\end{array}$}
+3\cdot 
\resizebox{2cm}{!}{$
	\begin{array}{ccccccccc}
	& & & &0& & & &\\[3mm]
	& & &0& &0& & &\\[3mm]
	& &0& &0& &1& &\\[3mm]
	&0& &1& &1& &0&\\[3mm]
	& &1& &0& &0& &\\[3mm]
	& & &0& &0& & &\\[3mm]
	& & & &0& & & &\\[3mm]
\end{array}$}
+
\resizebox{2cm}{!}{$
\begin{array}{ccccccccc}
& & & &0& & & &\\[3mm]
& & &0& &0& & &\\[3mm]
& &1& &0& &-1& &\\[3mm]
&0& &0& &0& &0&\\[3mm]
& &-1& &0& &1& &\\[3mm]
& & &0& &0& & &\\[3mm]
& & & &0& & & &\\[3mm]
\end{array}$}\\
&=
\resizebox{2cm}{!}{$
\begin{array}{ccccccccc}
& & & &1& & & &\\[3mm]
& & &1& &3& & &\\[3mm]
& &1& &3& &3& &\\[3mm]
&1& &3& &3& &1&\\[3mm]
& &1& &3& &1& &\\[3mm]
& & &1& &1& & &\\[3mm]
& & & &1& & & &\\[3mm]
\end{array}$}
-
\resizebox{2cm}{!}{$
\begin{array}{ccccccccc}
& & & &1& & & &\\[3mm]
& & &2& &1& & &\\[3mm]
& &2& &3& &1& &\\[3mm]
&1& &3& &3& &1&\\[3mm]
& &1& &3& &2& &\\[3mm]
& & &1& &2& & &\\[3mm]
& & & &1& & & &\\[3mm]
\end{array}$}
\end{align*}

$R=2\tilde{e}-B\tilde{d}$, so $R$ is redundant.
\[
\resizebox{2cm}{!}{$
	\begin{array}{ccccccccc}
	& & & &0& & & &\\[3mm]
	& & &0& &0& & &\\[3mm]
	& &0& &-1& &1& &\\[3mm]
	&0& &0& &0& &0&\\[3mm]
	& &1& &-1& &0& &\\[3mm]
	& & &0& &0& & &\\[3mm]
	& & & &0& & & &\\[3mm]
	\end{array}$}
=2\cdot
\resizebox{2cm}{!}{$
	\begin{array}{ccccccccc}
	& & & &0& & & &\\[3mm]
	& & &0& &0& & &\\[3mm]
	& &0& &0& &1& &\\[3mm]
	&0& &1& &1& &0&\\[3mm]
	& &1& &0& &0& &\\[3mm]
	& & &0& &0& & &\\[3mm]
	& & & &0& & & &\\[3mm]
	\end{array}$}
-
\resizebox{2cm}{!}{$
	\begin{array}{ccccccccc}
	& & & &0& & & &\\[3mm]
	& & &0& &0& & &\\[3mm]
	& &0& &1& &1& &\\[3mm]
	&0& &2& &2& &0&\\[3mm]
	& &1& &1& &0& &\\[3mm]
	& & &0& &0& & &\\[3mm]
	& & & &0& & & &\\[3mm]
	\end{array}$}
\]

\textbf{Degree 4:} It remains to show that $C\tilde{d}, D\tilde{d}, C T, DT$ and $Q$ lie in $\cI$. Since we work rationally, it will turn out we do not need new geometric generators.\medskip

Since $Q=A(BD-2AC)-C(B^2-4C)=AR-CS$ in $\cH_\ast$ and $R,CS\in \cI$, we have $Q\in \cI$.\medskip

$DT=2CS-AR\in \cI$.

\[
\resizebox{2.5cm}{!}{$
	\begin{array}{ccccccccccc}
	& & & & &0& & & & &\\[3mm]
	& & & &0& &0& & & &\\[3mm]
	& & &0& &1& &-1& & &\\[3mm]
	& &0& &0& &0& &0& &\\[3mm]
	&0& &1& &-2& &1& &0&\\[3mm]
	& &0& &0& &0& &0& &\\[3mm]
	& & &-1& &1& &0& & &\\[3mm]
	& & & &0& &0& & & &\\[3mm]
	& & & & &0& & & & &\\[3mm]
	\end{array}$}
=
2\cdot \resizebox{2.5cm}{!}{$
	\begin{array}{ccccccccccc}
	& & & & &0& & & & &\\[3mm]
	& & & &0& &0& & & &\\[3mm]
	& & &0& &0& &0& & &\\[3mm]
	& &0& &0& &0& &0& &\\[3mm]
	&0& &1& &-2& &1& &0&\\[3mm]
	& &0& &0& &0& &0& &\\[3mm]
	& & &0& &0& &0& & &\\[3mm]
	& & & &0& &0& & & &\\[3mm]
	& & & & &0& & & & &\\[3mm]
	\end{array}$}
- 
\resizebox{2.5cm}{!}{$
	\begin{array}{ccccccccccc}
	& & & & &0& & & & &\\[3mm]
	& & & &0& &0& & & &\\[3mm]
	& & &0& &-1& &1& & &\\[3mm]
	& &0& &0& &0& &0& &\\[3mm]
	&0& &1& &-2& &1& &0&\\[3mm]
	& &0& &0& &0& &0& &\\[3mm]
	& & &1& &-1& &0& & &\\[3mm]
	& & & &0& &0& & & &\\[3mm]
	& & & & &0& & & & &\\[3mm]
	\end{array}$}
\]

$4CT=2DS-ABS+B^2T$, so $CT\in \cI$.
\[
4\cdot\resizebox{2.5cm}{!}{$
	\begin{array}{ccccccccccc}
	& & & & &0& & & & &\\[3mm]
	& & & &0& &0& & & &\\[3mm]
	& & &0& &0& &0& & &\\[3mm]
	& &0& &1& &-1& &0& &\\[3mm]
	&0& &0& &0& &0& &0&\\[3mm]
	& &0& &-1& &1& &0& &\\[3mm]
	& & &0& &0& &0& & &\\[3mm]
	& & & &0& &0& & & &\\[3mm]
	& & & & &0& & & & &\\[3mm]
	\end{array}$}
=
2\cdot \resizebox{2.5cm}{!}{$
	\begin{array}{ccccccccccc}
	& & & & &0& & & & &\\[3mm]
	& & & &0& &0& & & &\\[3mm]
	& & &0& &0& &0& & &\\[3mm]
	& &0& &1& &-2& &1& &\\[3mm]
	&0& &0& &0& &0& &0&\\[3mm]
	& &1& &-2& &1& &0& &\\[3mm]
	& & &0& &0& &0& & &\\[3mm]
	& & & &0& &0& & & &\\[3mm]
	& & & & &0& & & & &\\[3mm]
	\end{array}$}
- 
\resizebox{2.5cm}{!}{$
	\begin{array}{ccccccccccc}
	& & & & &0& & & & &\\[3mm]
	& & & &0& &0& & & &\\[3mm]
	& & &0& &0& &0& & &\\[3mm]
	& &1& &-1& &-1& &1& &\\[3mm]
	&0& &0& &0& &0& &0&\\[3mm]
	& &1& &-1& &-1& &1& &\\[3mm]
	& & &0& &-0& &0& & &\\[3mm]
	& & & &0& &0& & & &\\[3mm]
	& & & & &0& & & & &\\[3mm]
	\end{array}$}
+ 
\resizebox{2.5cm}{!}{$
	\begin{array}{ccccccccccc}
	& & & & &0& & & & &\\[3mm]
	& & & &0& &0& & & &\\[3mm]
	& & &0& &0& &0& & &\\[3mm]
	& &-1& &-1& &1& &1& &\\[3mm]
	&0& &0& &0& &0& &0&\\[3mm]
	& &-1& &-1& &1& &1& &\\[3mm]
	& & &0& &-0& &0& & &\\[3mm]
	& & & &0& &0& & & &\\[3mm]
	& & & & &0& & & & &\\[3mm]
	\end{array}$}
\]

 Finally, since $2Cd=Be$ and $2Dd=ABd$ in $\DR_\ast$, the elements $2C\tilde{d}-B\tilde{e}$ and $2D\tilde{d}-AB\tilde{d}$ lie in the kernel of $s$. As we have already found generators for all of $(\ker s)_{\geq 3}$ and $AB\tilde{d}, B\tilde{e}\in \cI$, we obtain $C\tilde{d}, D\tilde{d}\in \cI$. \end{proof}
\begin{rem} (Integral version)
As the proof shows, the integral version of the theorem is true in dimension $3$ and to make the same proof work for any dimension it would suffice to invert $2$. The general case depends on the geometric realization of the integral generators for $K'$ in dimension $4$, compare Problem \ref{prob: 4-folds for integral statement}. We are optimistic that further progress in the classification of left-invariant complex structures on fixed $8$-dimensional nilpotent Lie-algebras, as begun in \cite{latorre_complex_2020}, combined with a systematic study of the Dolbeault cohomology stratification of their moduli space, will eventually settle this.
\end{rem}
\begin{rem}
The following is another way of showing $AT\in\cI$: Denote by $Z$ the twistor space of a complex $2$-torus $\mathbb T$ equipped with Ricci-flat Kähler metric. There is a diffeomorphism $Z\cong \mathbb \mathbb T\times\CP^1$. However, using the results of \cite{eastwood-singer}, the reader may verify:
\[
2AT+2BT+2\tilde{e}+B\tilde{d}=h(Z-\mathbb T\times\CP^1).
\]
We chose to use nilmanifolds throughout since they have vanishing Chern classes and work for all generators.
\end{rem}

Adding what needs to be added, the rest of the proof of Theorem \ref{thmintro: top inv} proceeds almost literally as in the K\"ahler case. For the reader's convenience, we give the full argument.
\begin{prop}\label{prop: ideal diffeomorphic, Hodge}
	In degrees $\geq 3$, $\cI$ coincides with the ideal generated by differences of diffeomorphic complex manifolds, and $\cI\cap\ker(\sigma)$ coincides with the ideal generated by differences of orientation preservingly homeomorphic (or diffeomorphic) manifolds.
\end{prop}
\begin{proof}
	As we have seen in the proof of Theorem \ref{thm: ideal homeomorphic, Hodge}, except for the multiples of $S$, all non-redundant generators of $\cI_{\geq 3}$ are given by differences of two copies of the same nilmanifold equipped with different left-invariant complex structures, which are trivially (even orientation-preservingly) isomorphic. For multiples of $S$, recall that $\cH_\ast$ can be generated by $\CP^1,\CP^2$ an elliptic curve $E$ and a Hopf surface $H$.  By a result of Wall \cite{wall_simply-connected_1964}, the manifolds $X_i$ and $Y_i$ are $h$-cobordant. Thus, by the $h$-cobordism theorem \cite{smale_structure_1962}, $\CP^k\times X_i$ and $\CP^k\times Y_i$ are diffeomorphic. Products with $E\cong S^1\times S^1$ and $H\cong S^1\times S^3$ are handled by the following lemma from \cite{kervaire_theoreme_1965}:
	\begin{lem}\label{lem: s-cobordism circle}
	Let $M,N$ be $h$-cobordant manifolds of dimension $\geq 5$. Then $M\times S^1$ and $N\times S^1$ are diffeomorphic.
	\end{lem}
	To take into account orientations, we first note that all generators of $\cI_{\geq 3}$, except products of the differences of orientation reversingly homeomorphic complete intersections $X_i-Y_i$ with pure powers of $\CP^2$, have vanishing signature, so that $\cI_{\geq 3}\cap\ker\sigma$ is generated by products of $X_i-Y_i$ with monomials containing at least one factor of $\CP^1$, $E$ and $H$ and those generators coming from differences of nilmanifolds. As already remarked above, the latter are differences of orientation preservingly diffeomorphic manifolds. On the other hand, since $\CP^1, E$ and $H$ admit orientation reversing self-diffeomorphisms, we may deduce that also products of $X_i$ and $Y_i$ with any of these three manifolds are orientation preservingly diffeomorphic.\end{proof}

\begin{prop}\label{prop: ideal homeo diffeo, chern}
	Let $\cJ\subseteq \cH\cC_\ast\otimes\Q$ be the ideal generated by differences of homeomorphic complex manifolds. In degrees $\geq 3$, it coincides with the ideal generated by differences of diffeomorphic complex manifolds and with the kernel $K$ of the composition
	\[
	\cH\cC_\ast\otimes\Q\overset{(\chi,h^{0,0})}{\longrightarrow}\Q[z]\times \Q[z].
	\]
\end{prop}

\begin{proof}
Let $\beta_1=\CP^1,\beta_2,...\beta_i,...$ be a generating sequence for $\Omega_\ast^U\otimes \Q$ by compact complex manifolds. Note that we do not assume that $\beta_i=\CP^i$ for $i\geq 2$. Given any element $(h(X),[X])\in\cH\cC_n$, write $[X]=[P]$ for some polynomial $P$ in $\beta_1,...,\beta_n$. Then $(h(X)-h(P),0)\in\cH\cC_n$ and $h(X-P)$ is in the kernel of the Hirzebruch genus, which may be generated by an elliptic curve $E$ and a Hopf surface $H$, which map to zero in $\Omega_\ast^U$. From this, we see that $\cH\cC_\ast\otimes\Q$ is generated by an elliptic curve $E$, a Hopf surface $H$ and the $\beta_i$, i.e. there is a commutative diagram
\[
\begin{tikzcd}
\Q[E,H,\beta_1,\beta_2,...]\ar[rd]\ar[d]&\\
\cH\cC_\ast\otimes\Q\ar[r]&\Q[z]\times\Q[z].
\end{tikzcd}
\]

Clearly, $\cJ\subseteq K$. On the other hand, by \cite[Thm. 10]{kotschick_topologically_2012}, it is possible to choose a basis sequence $\beta_1=\CP^1,\beta_2,...$ such that $\beta_i\in\cJ$ for $i\geq 2$. For such a basis sequence, the above diagram may be reduced to
\[
\begin{tikzcd}
\Q[E,H,\beta_1]\ar[rd]\ar[d]&\\
\cH\cC_\ast\otimes\Q/(\beta_i)_{i\geq 2}\ar[r]&\Q[z]\times\Q[z].
\end{tikzcd}
\]
Write $B=\beta_1-E$ and $H'=E^2-H$. Then, since $\chi(B)=2$, $h^{0,0}(B)=0$ and $\chi(E)=0$, $h^{0,0}(E)=1$, $\chi(H')=0=h^{0,0}(H')$, the surjective map
$\Q[E,H',B]=\Q[E,H,\beta_1]\to\Q[z]\times\Q[z]$ has kernel generated by $EB$ and $H'$. Thus $K=\langle EB,H',\beta_2,\beta_3,...\rangle$. Since both $EB$ and $H'$ have vanishing Chern classes, their images in $\cH\cC_\ast\otimes\Q$ lie entirely in the ideal of elements of the form $(a,0)\in \cH\cC_\ast\subseteq\cH_\ast\times\cC_\ast$, which, by Remark \ref{rem: HC diagonal}, is identified with the kernel of the Hirzebruch genus $\chi_\ast:\cH_\ast\otimes\Q\to\cH ir_\ast\otimes\Q$. In particular, they lie in the kernel of $\chi:\cH_\ast\otimes\Q\to\Q[z]$. Since they also lie in the kernel of $h^{0,0}$, by Theorem \ref{thm: ideal homeomorphic, Hodge}, all multiples in degrees $\geq 3$ lie in $\cJ$. This completes the proof of $K_{\geq 3}=\cJ_{\geq 3}$.

To take into account diffeomorphisms instead of homeomorphisms, note that the generators $\beta_i$ constructed in \cite{kotschick_topologically_2012} are in fact differences of diffeomorphic projective varieties as soon as $i\geq 3$, and the same holds for $\beta_1\cdot\beta_2$ and $\beta_2\cdot\beta_2$. Since $\beta_2$ is a difference of orientation-reversingly homeomorphic simply connected manifolds, it follows as in the proof of Proposition \ref{prop: ideal diffeomorphic, Hodge} from the $h$-cobordism theorem and Lemma \ref{lem: s-cobordism circle} that $E\times \beta_2$ and $H\times\beta_2$ are differences of diffeomorphic manifolds. As we have seen, all other generators for $\cJ_{\geq 3}$ may be taken to be differences of diffeomorphic nilmanifolds with left invariant complex structures.
\end{proof}

\begin{prop}\label{prop: ideal oriented homeo diffeo, chern}
Let $\cJ\Oh\subseteq\cH\cC_\ast\otimes\Q$ be the ideal generated by differences of orientation preservingly homeomorphic manifolds. In degrees $\geq 3$, it coincides with the ideal generated by differences of orientation preservingly diffeomorphic manifolds and with the kernel $K$ of the map
\[
\cH\cC_\ast\otimes\Q\longrightarrow\Q[z]\times\Q[z]\times\Omega_{2\ast}^{SO}\otimes\Q,
\]
given by $(\chi,h^{0,0})$ in the first two components and the forgetful map in the third.
\end{prop}

\begin{proof}
Clearly, the ideal generated by differences of orientation preservingly diffeomorphic manifolds is contained in $\cJ\Oh$ which is in turn contained in the kernel. For the reverse inclusions, let us use the same basis sequence as in the previous proof and consider the diagram
\[
\begin{tikzcd}
\Q[E,H,\beta_1,\beta_2,...]\ar[rd]\ar[d]&\\
\cH\cC_\ast\otimes\Q\ar[r]&\Q[z]\times\Q[z]\times\Omega_{2\ast}^{SO}\otimes\Q.
\end{tikzcd}
\]
The elements $E,\beta_1$ have trivial Pontryagin numbers and we see, as before, that the subring of polynomials in these two elements surjects onto the first two factors. On the other hand, the polynomials in the $\beta_{2i}$ generate $\Omega_{2\ast}^{SO}\otimes\Q$ and have zero image in the first two factors. Thus, the diagonal map is surjective and its kernel is the ideal generated by the elements $\beta_i$ for $i$ odd, $\beta_1\cdot\beta_j$, $E\cdot \beta_j$, $H\cdot \beta_j$ for $j$ even and $EB, H'$ as above. Using \cite[Thm. 7]{kotschick_topologically_2012} for the $\beta_1\cdot\beta_{i-1},\beta_i\in\cJ\Oh$ for odd $i\geq 3$, and, as before, that $E$ and $H$ admit orientation reversing self-diffeomorphisms, all these elements are representable by differences of orientation preservingly diffeomorphic manifolds.
\end{proof}

\begin{proof}[Proof of Theorem \ref{thmintro: top inv}]
	A rational linear combination of Hodge and Chern numbers of $n$-dimensional manifolds is a linear map $\cH\cC_n\to\Q$. It is an unoriented homeomorphism invariant precisely if it vanishes on $\cJ$. Since by Proposition \ref{prop: ideal homeo diffeo, chern}, for $n\geq 3$ the ideal generated by differences of diffeomorphisms coincides with $\cJ_n$, the linear combinations invariant under homeomorphism are exactly those invariant under diffeomorphism.
	
	Further, it is known that all of the listed quantities in Theorem \ref{thmintro: top inv} are homeomorphism invariants, so it remains to show the converse. By Proposition \ref{prop: ideal homeo diffeo, chern}, any linear map $h:\cH\cC_n\to\Q$, $n\geq 3$, which is a (not necessarily orientation preserving) homeomorphism  invariant (i.e. it vanishes on $\cJ_n$), factors over $(h^{0,0},\chi)$, i.e. it can be written as a linear combination of Euler characteristic and connected components only. 
	
	The proof for the orientation preserving case is the same, using Proposition \ref{prop: ideal oriented homeo diffeo, chern}. The statements involving Betti numbers follow from Theorem \ref{thmintro: HDRC rels}.
\end{proof}

\section{Refined Betti numbers}\label{sec: refined Betti numbers}
The refined Betti numbers were originally introduced in \cite{stelzig_structure_2018} from the point of view of the structure theory of double complexes. We give here a reminder and more elementary treatment, independent of \cite{stelzig_structure_2018}.  The connection with \cite{stelzig_structure_2018} will be picked up again in Section \ref{sec: univ ring}.\medskip

Let $X$ be a compact complex manifold of dimension $n$. The $k$-th complex de Rham cohomology $H_{dR}^k(X)$ is naturally equipped with two filtrations, namely the Hodge filtration and its conjugate:
\begin{align*}
F^p:=F^pH_{dR}^k(X):=\left\{[\omega]~\bigg|~ \omega\in\bigoplus_{\substack{r+s=k\\r\geq p}}A_X^{r,s}\right\}\\
\Fbar^q:=\Fbar^qH_{dR}^k(X):=\left\{[\omega]~\bigg|~ \omega\in\bigoplus_{\substack{r+s=k\\s\geq q}}A_X^{r,s}\right\}
\end{align*}
If $X$ satisfies the $\del\delbar$-Lemma, for example for $X$ K\"ahler, it is a standard result that $F$ and $\Fbar$ induce a pure Hodge structure on the de Rham cohomology, i.e. that
\[
H_{dR}^k(X)=\bigoplus_{p+q=k}F^p\cap \Fbar^q.
\]

For general $X$, this is no longer true. This motivates the following definition:

\begin{definition}
	The \textbf{total filtration} on $H_{dR}^k(X)$ is defined by
	\[
	F_{tot}^l:=F_{tot}^l H_{dR}^k(X):=\sum_{p+q=l} F^p\cap \Fbar^q.
	\]
	Denote by $(H_{dR}^k(X))_l:=\gr_F^lH_{dR}^k(X):=F_{tot}^l / F_{tot}^{l+1}$ the $l$-th associated graded and write 	
	\[
	H_k^{p,q}:=\im\left(F^p\cap \Fbar^q\rightarrow (H_{dR}^k(X))_{p+q}\right)=\frac{F^p\cap\Fbar^q}{F^{p+1}\cap \Fbar^q+F^p\cap\Fbar^{q+1}}.
	\]
\end{definition}

\begin{rem}
	The total filtration is invariant under conjugation and may therefore also be considered as a filtration on the real cohomology. 
\end{rem}
\begin{rem}
	In the following, we will mainly care about the spaces $H_k^{p,q}$. The reader preferring a quicker but less symmetric definition may verify the existence of a natural isomorphism $H_k^{p,q}\cong \gr_F^p\gr_{\Fbar}^q H_{dR}^k(X)$.
\end{rem}

By definition, we obtain:
\begin{lem}[Weak Hodge decomposition]
	\[
	(H_{dR}^k(X))_l=\bigoplus_{p+q=l}H_k^{p,q}
	\]
\end{lem}

The spaces occurring in this decomposition for $p+q\neq k$ therefore measure the defect of purity. 
\begin{definition}
The numbers \[
b_k^{p,q}:=b_k^{p,q}(X):=\dim H_k^{p,q}(X)
\]
are called the \textbf{refined Betti numbers} of $X$.
\end{definition}

Since an associated graded of a filtered vector space has the same dimension as the original one, we obtain $b_k=\sum_{p,q\in\Z}b_k^{p,q}$, explaining the name. We note that unlike the Betti numbers, their refined counterparts are sensitive to the complex structure. By construction, $b_k=\sum_{p+q=k} b_k^{p,q}$ if and only if $F$ and $\Fbar$ induce a pure Hodge structure. Outside certain triple degrees, the $b_k^{p,q}$ have to vanish (see Proposition \ref{prop: obvious linear relations between refined Betti numbers} below). One may picture the $b_k^{p,q}$ as giving a three-dimensional analogue of the Hodge diamond, or equivalenty a diamond for each $k$, e.g. for $n=2$:
\tiny
\[
\begin{array}{ccccc}
\begin{array}{c}
b_0^{0,0}
\end{array}&

\begin{array}{ccc}
&b_1^{1,1}&\\[3mm]
b_1^{1,0}&&b_1^{0,1}\\[3mm]
&b_1^{0,0}&
\end{array}&
\begin{array}{ccccc}
&&b_2^{2,2}&&\\[3mm]
&b_2^{2,1}&&b_2^{1,2}&\\[3mm]
b_2^{2,0}&&b_2^{1,1}&&b_2^{0,2}\\[3mm]
&b_2^{1,0}&&b_2^{0,1}&\\[3mm]
&&b_2^{0,0}&&
\end{array}
&
\begin{array}{ccc}
&b_1^{2,2}&\\[3mm]
b_3^{2,1}&&b_3^{1,2}\\[3mm]
&b_3^{1,1}&
\end{array}
&
\begin{array}{c}
b_4^{2,2}
\end{array}
\end{array}
\]
\normalsize

For K\"ahler (or $\del\delbar$-)manifolds, there is an isomorphism $H_{\delbar}^{p,q}(X)\cong H_{p+q}^{p,q}(X)$, so the refined Betti numbers are a different generalisation of the Hodge numbers on K\"ahler manifolds.\medskip

For an explicit understanding of the spaces $H_k^{p,q}$, the following Lemma is useful. It is the general version of the well-known fact that for $X$ K\"ahler (or more generally carrying a pure Hodge structure), the spaces $F^p\cap F^q\subseteq H_{dR}^k(X)$ with $p+q=k$ consist of classes representable by pure-type forms.

\begin{lem}\label{lem: spaces FcapFbar}
	Fix some $k\in\Z_{\geq 0}$. The subspaces $F^p\cap \Fbar^q\subseteq H_{dR}^k(X)$ allow the following explicit description:
	
	If $p+q\geq k$:
	\[
	F^p\cap \Fbar^q=\left\{\parbox{9cm}{classes which admit for any $(r,s)\in\{(k-q,q),...,(p,k-p)\}$ a representative $\omega\in A^{r,s}$.}\right\}
	\]
	If $p+q\leq k$:
	\[
	F^p\cap \Fbar^q=\left\{\parbox{5.5cm}{classes which admit a representative ${\omega=\sum_{j=p}^{k-q}\omega_{j,k-j}}$ with $\omega_{r,s}\in A^{r,s}$.}\right\}
	\]
\end{lem}
\begin{proof}
	For an element $\omega\in A_X$, let us denote by $\omega^{r,s}$ its component in bidegree $(r,s)$. By definition, a class $\mathfrak{c}\in H^k_{dR}(X)$ is in $F^p\cap \Fbar^q$ if it has a representative $\omega=\sum_{r+s=k}\omega^{r,s}$ with $\omega^{r,s}=0$ for $r<p$ and another one $\omega'=\sum_{r+s=k}{\omega'}^{r,s}$ with ${\omega'}^{r,s}=0$ for $s<q$. So the inclusions from right to left are immediate and it remains to show the converse.\medskip
	
	Let $\omega,\omega'$ be two representatives of a class $\mathfrak{c}\in F^p\cap \Fbar^q$ as above and let $\eta=\sum_{r,s\in\Z}\eta^{r,s}$ be a form of total degree $k-1$ with $\omega=\omega'+d\eta$. This gives us a sequence of equations
	
	\[
	\omega^{r,s}={\omega'}^{r,s}+\del\eta^{r-1,s}+\delbar\eta^{r,s-1}\tag{$\ast$}.
	\]
	
	Set $\tilde{p}:=\max\{p,k-q\}$ and $\tilde{q}:=\max\{q,k-p\}$. Replacing $\omega$ by the cohomologous $\omega-\sum_{i\geq\tilde{p}}d\eta^{i,k-i-1}$ and $\omega'$ by $\omega'-\sum_{i\geq\tilde{q}}d\eta^{k-i-1,i}$, we may assume that $\eta^{r,s}=0$ for $r\geq \tilde{p}$ or $s\geq \tilde{q}$, $\omega^{r,s}=0$ for $r\not\in[p,\tilde{p}]$ and ${\omega'}^{r,s}=0$ for $r\not\in[q,\tilde{q}]$.\medskip 
	
	Now we distinguish two cases. If $\tilde{p}=k-q$, i.e. $k\geq p+q$, we are done. If $\tilde{p}=p>k-q$, the element $\omega=\omega^{p,k-p}$ is pure of bidegree $(p,k-p)$ and $\omega'=\omega^{k-q,q}$ pure of bidegree $(k-q,q)$. By $(\ast)$, we obtain $\omega=\del\eta^{p-1,k-p}$ which is cohomologous to the pure element $-\delbar\eta^{p-1,k-p}$. Applying the same reasoning over and over again, we obtain representatives for $\mathfrak{c}$ that are pure in degrees $(k-q,q),...,(p,k-p)$. \end{proof}

\begin{prop}\label{prop: obvious linear relations between refined Betti numbers}
The refined Betti numbers of an $n$-dimensional compact complex manifold $X$ satisfy the following universal relations:
\begin{enumerate}
\item[(B1)] Bounded support: One has $b_k^{p,q}(X)=0$ unless $0\leq p,q\leq k$.
\item[(B2)] Conjugation symmetry: $b_k^{p,q}(X)=b_k^{q,p}(X)$ for all $p,q,k\in\Z$.
\item[(B3)] Serre symmetry: $b_k^{p,q}(X)=b_{2n-k}^{n-p,n-q}(X)$.
\item[(B4)] Boundary case vanishing: $b_1^{1,1}(X)=0$ and $b_n^{n,p}(X)=0$ for $p>0$.
\end{enumerate}
\end{prop}

\begin{proof}
	$(B1)$ and $(B2)$ follow directly from the definitions.
	
	For $(B3)$, consider the map of double complexes $A_X\to DA_X$, where $DA_X$ is the dual double complex, given by $(DA_X)^{p,q}:=\Hom (A_X^{n-p,n-q},\C)$ with total differential
	\begin{align*}		
		d_{DA_X}^{p+q}:=(\varphi\mapsto (-1)^{p+q+1}\varphi\circ d^{2n-p-q-1}).
	\end{align*}
	Integration $A_X\to DA_X$, $\omega\mapsto\int_X\omega\wedge\_$ defines a map of double complexes, and hence of the induced spectral sequences.
	Noting that for the first page of the spectral sequence associated with the column filtration of $DA_X$ one has
	\[
	E_1^{p,q}(DA_X)=(E_1^{n-p,n-q}(A_X))^\vee=(H_{\delbar}^{n-p,n-q}(A_X))^\vee,
	\]
	 by Serre duality \cite{serre_theoreme_1955}, integration gives an isomorphism between the first pages (and hence on all later ones). Therefore also
	\[
	E_r^{p,q}(A_X)\cong(E_r^{n-p,n-q}(A_X))^\vee
	\]
	for all $r\geq 1$. Thus:
	\[
	\gr^p_F H^k_{dR}(X)=E_{\infty}^{p,k-p}(A_X)\cong(E_{\infty}^{n-p,n-k+p}(A_X))^\vee=(\gr^{n-p}_FH_{dR}^{2n-k}(X))^\vee.
	\]
	
	As a consequence, integration induces isomorphisms
	\[
	F^pH^k_{dR}(X)\cong F^p H_{dR}^k(DA_X)\cong (H^{2n-k}_{dR}(X)/F^{n-p+1})^\vee
	\]
	for all $p$. By conjugation, the same holds for the conjugate filtration $\Fbar$. Thus, $H_k^{p,q}(X)\cong H_k^{p,q}(DA_X)\cong(H_{2n-k}^{n-p,n-q}(X))^\vee$.

	$(B4)$ follows from well-known arguments (see e.g. \cite[Ch. IV Lem. 2.1-2.3]{barth_compact_1984}). We sketch the argument for the reader's convenience: For the first part of $(B4)$, pick an element $\mathfrak{c}\in H_1^{1,1}=F^1\cap\Fbar^1$ and two representatives $\omega^{1,0}\in A_X^{1,0}$ and $\omega^{0,1}\in A_X^{0,1}$. Thus, $\omega^{0,1}-\omega^{1,0}=df$, i.e. $\omega^{0,1}=\delbar f$. Hence $\del\delbar f=d\omega^{0,1}=0$ i.e. $f$ is pluriharmonic. But since $X$ is compact, $f$ has to be constant and therefore $\mathfrak{c}=0$. Since $\mathfrak{c}$ was arbitrary, this shows $H_1^{1,1}=0$.
	
	As for the second part, let $\mathfrak{c}\in F^n\cap  \Fbar^p\subseteq H_{dR}^n(X)$ with $p>0$. By Lemma \ref{lem: spaces FcapFbar}, there exists a representative $\omega$ of type $(n,0)$ but also another one, say $\eta$, of type $(n-p,p)$. Writing out the equation $\omega-\eta=d(\theta)$ by bidegrees, we see that $\omega$ is a $\del$-exact holomorphic $n$-form. Using Stokes' theorem, this implies $\omega=0$. Thus, in particular $H_n^{n,p}=0$ for $p>0$.
\end{proof}

\begin{prop}\label{prop: refined Kuenneth}
For a product $X\times Y$ of compact complex manifolds, the K\"unneth formula is strictly compatible with the Hodge filtration and its conjugate. In particular, the following relation holds:
\[
b_k^{p,q}(X\times Y)=\sum_{\substack{k_1+k_2=k\\p_1+p_2=p\\q_1+q_2=q}}b_{k_1}^{p_1,q_1}(X)\cdot b_{k_2}^{p_2,q_2}(Y)
\]
\end{prop}
Recall that a map of filtered vector spaces $f:(V,F^\Cdot)\to(W,F^\Cdot)$ is called strict if $f(F^p)=\im f \cap F^p$ for all $p$.
\begin{proof}
The K\"unneth isomorphism in de Rham cohomology is induced by the map $\pi_X^*\otimes\pi_Y^*:A_X\otimes A_Y\longrightarrow A_{X\times Y}$ given by the two pullback maps. This is a map of double complexes and the grading on $A_X \otimes A_Y$ is the tensor product grading. The K\"unneth formula in Dolbeault cohomology states that it induces an isomorphism in Dolbeault cohomology, i.e. the first page of the Fr\"olicher spectral sequence. Hence, the induced maps on all later pages are isomorphisms as well. In particular this holds for the page $E_\infty$. But that page is the graded vector space associated with the Hodge filtration on $H_{dR}$ and a map of filtered vector spaces is a strict isomorphism if and only if the map of the associated graded is an isomorphism. The result for the conjugate filtration follows since $\pi_X^*\otimes\pi_Y^*$ is, as all maps of geometric origin, conjugation invariant.
\end{proof}

\section{The refined de Rham ring}\label{sec: refined de Rham ring}
By Proposition \ref{prop: refined Kuenneth}, we obtain a ring homomorphism
\begin{align*}
rb:\CM_\ast&\longrightarrow \Z[x,y,h,z]\\
[X]&\longmapsto rb(X):=\sum_{p,q,k}b_k^{p,q}(X)x^p y^q h^k z^{\dim X}.
\end{align*}
We denote the image by $\RB_\ast$ and call it the \textbf{refined de Rham ring}.

\begin{definition}\label{def. formal de Rham ring}
$\RB^{form}_\ast$ and $\RB'_\ast$  are the graded subrings of $\Z[x,y,h,z]$ defined by:
\[
\RB^{form}_n:=\left\{\left(\sum_{p,q,k\geq 0} b_{k}^{p,q} x^py^qh^k\right)\cdot z^n\in\Z[x,y,h,z]~\bigg|~ \substack{b_k^{p,q}\text{ subject to conditions}\\\text{(B1) -- (B4) of Prop. \ref{prop: obvious linear relations between refined Betti numbers}}}\right\}
\] and 
\[
\RB'_n:=\left\{\left(\sum_{p,q,k\geq 0} b_{k}^{p,q} x^py^qh^k\right)\cdot z^n\in\Z[x,y,h,z]~\bigg|~ \substack{b_k^{p,q}\text{ subject to conditions}\\\text{(B1) -- (B3) of Prop. \ref{prop: obvious linear relations between refined Betti numbers}}}\right\}.
\]
\end{definition}

\begin{thm}\label{thm: surjection to RB'} There is a surjective map of rings
\begin{align*}
\Phi:\Z[A,B,C,L,M]&\longrightarrow \RB'_\ast\\
A&\longmapsto(1+xyh^2)z\\
B&\longmapsto(xh+yh)z\\
C&\longmapsto xyh^2z^2\\
L&\longmapsto(h+x^2y^2h^3)z^2\\
M&\longmapsto(xyh+h)z.
\end{align*}

Setting $|A|=|B|=|M|=1$ and $|C|=|L|=2$, it is compatible with the grading. The kernel is given by the principal ideal $I$ generated by

\[
AML-L^2-CM^2-CA^2+4C^2.
\]

\end{thm}

The degree $n$ part $\RB'_n$ is a free $\Z$-module. The proof of surjectivity works by writing down preimages for every element in a basis. The computation of the kernel will follow from a rank counting argument. We will need two preparatory Lemmas:

\begin{lem}
The degree $n$ part of $\Z[A,B,C,L,M]/I$ has rank 
\[
r_n:=\rk (\Z[A,B,C,L,M]/I)_n=\left\lfloor\frac{2n^3+9n^2+16n+12}{12}\right\rfloor.
\]
\end{lem}

\begin{proof}
Recall that $A,B,M$ have degree $1$ and $C,L$ degree $2$. Hence, the generating function for the numbers $s_n:=\rk\Z[A,B,C,L,M]_n$ is given by 

\begin{align*}
\sum_{n=0}^{\infty}s_nt^n&=\left(\sum_{i=0}^\infty t^i\right)^3\cdot\left(\sum_{i=0}^\infty t^{2i}\right)^2\\
&=\left(\sum_{i=0}^\infty\binom{i+2}{2}t^i\right)\cdot\left(\sum_{i=0}^\infty\binom{i+1}{1}t^{2i}\right)\\
&=\sum_{n=0}^\infty\left(\sum_{i+2j=n}\binom{i+2}{2}(j+1)\right)t^n.
\end{align*}
Since $I$ is a principal ideal generated by an element of degree $4$, one has $r_n=s_n-s_{n-4}$. In particular, $r_n=s_n$ for $n=0,1,2,3$ and in these cases the formula of the lemma is easily checked directly (the values being $r_0=1$, $r_1=3$, $r_2=8$ and $r_3=16$). The other cases follow by induction using 
\[
s_n=s_{n-2}+\sum_{i+2j=n}\binom{i+2}{2}.
\]\end{proof}

Fix some degree $n\in\Z_{\geq 0}$. For $k\in \{0,...,2n\}$ and $p,q\in [\max(k-n,0),\min(k,n)]\cap\Z$, define the polynomial
\[
Sym^{p,q}_k(n):= \alpha^{p,q}_{k,n}\cdot(x^py^qh^k+x^qy^ph^k+x^{n-p}y^{n-q}h^{2n-k}+x^{n-q}y^{n-p}h^{2n-k})z^n\in \RB'_n,
\]
where $\alpha^{p,q}_{k,n}\in\{1,\frac{1}{2},\frac{1}{4}\}$ is the number of distinct monomials occurring divided by four. To explain the need for $\alpha$, there are two involutions acting on the degree $n$ part of $\Z[x,y,h,z]$: the one flipping along the diagonal, i.e. $x^py^q\mapsto x^qy^p$, and the one flipping along the antidiagonal, i.e. $x^py^qh^kz^n\mapsto x^{n-p}y^{n-q}h^{2n-k}z^n$. These combine into an action of $(\Z/2\Z)^2$. Conditions {\it(B2)} and {\it(B3)} ensure that elements in $\RB'_\ast$ are invariant with respect to this action. The polynomial $Sym^{p,q}_k(n)$ is the sum of the elements in the orbit of $x^py^qh^kz^n$.

\begin{lem}
The following polynomials constitute a basis for $\RB'_n$:
\[
Sym^{p,q}_k(n)\text{ for }0\leq k\leq n-1\text{ and }0\leq q\leq p\leq k,
\]
\[
Sym^{p,q}_n(n)\text{ for } 0\leq q\leq p\leq n\text{ and } p+q\leq n.
\] 
The rank of $\RB'_n$ is 
\[
\rk \RB'_n=\left\lfloor\frac{2n^3+9n^2+16n+12}{12}\right\rfloor.
\]
\end{lem}
\begin{proof}
An arbitrary element $P\in \RB'_n$ is a sum of monomials subject to condition {\it(B1)}. By the symmetry conditions {\it(B2)} and {\it(B3)}, for every monomial $a\cdot x^py^qh^dz^n$ occurring in $P$, every monomial in $Sym^{p,q}_k(n)$ occurs with coefficient $a$. Hence the $Sym^{p,q}_k(n)$ with only condition {\it(B1)} as restriction on $p,q,k$ form a generating set. It remains to reduce the redundancy caused by the symmetries $Sym^{p,q}_k(n)=Sym^{q,p}_k(n)=Sym^{n-p,n-q}_{2n-k}(n)$.\medskip

For fixed $k$ and $n$, one may picture the allowed monomials $x^py^qh^kz^n$ occurring as summands in elements of $\RB'_n$ in a `Hodge diamond', e.g. for $k=2$, $n\geq 2$, the following (up to multiplication with $h^2z^n$):
\[
\begin{array}{ccccc}
&&x^2y^2&&\\[3mm]
&x^2y^1&&x^1y^2&\\[3mm]
x^2&&xy&&y^2\\[3mm]
&x&&y&\\[3mm]
&&1&&
\end{array}
\]
By the symmetry under exchange of $x$ and $y$, one only has to consider the left half (including the central column) of each diamond. Similarly, by the symmetry along $x^py^qh^kz^n\mapsto x^{n-p}y^{n-q}h^{2n-k}z^n$, one only has to consider the first $n-1$ of the diamonds and the lower half (including central row) of the $n$-th diamond, and there are no further redundancies. The $Sym^{p,q}_k(n)$ in the statement of the lemma are exactly those such that $x^py^qh^kz^n$ lies in both of these regions.\medskip

It remains to show the statement of the rank. For the $k$-th diamond, where $0\leq k\leq n-1$, one obtains $T_{k+1}=1+2+...+(k+1)=\frac{(k+2)(k+1)}{2}$ polynomials. For the $n$-th diamond, one only has to count monomials in the lower left quarter. If $n$ is odd (resp. even), this amounts to summing the even (resp. odd) numbers between $0$ and $n+1$, yielding $(\frac{n+1}{2})(\frac{n+3}{2})$ (resp. $(\frac{n+2}{2})^2$).

Since the sum of the first $r$ triangular numbers is $\frac{r(r+1)(r+2)}{6}$, one obtains a total dimension of

\[
\rk \RB'_n=\frac{n(n+1)(n+2)}{6}+\frac{(n+1)(n+3)}{4}=\frac{2n^3+9n^2+16n+9}{12}
\]
if $n$ is odd and
\[
\rk \RB'_n=\frac{n(n+1)(n+2)}{6}+\left(\frac{n+2}{2}\right)^2=\frac{2n^3+9n^2+16n+12}{12}
\]
for $n$ even.
\end{proof}

\begin{proof}[Proof of Theorem \ref{thm: surjection to RB'}]
Our goal is thus to show that all the basis elements $Sym^{p,q}_k(n)$ lie in the subring generated by the images of $A,B,C,M,L$ above. We do this by induction on $n$.\medskip

For $n=1$, one has $A\mapsto Sym_0^{0,0}(1)$, $B\mapsto Sym_1^{1,0}(1)$ and $M\mapsto Sym_1^{0,0}(1)$ by definition. From now on, let $n\geq 2$.\medskip

Since $\Phi(C)=xyh^2z^2=Sym_{2}^{1,1}(2)$ consists of one monomial, one has 
\[
Sym^{p,q}_k(n)\cdot \Phi(C)=Sym^{p+1,q+1}_{k+2}(n+2).
\]
Thus, if $p,q\not\in\{0,k\}$ (and hence necessarily $k\geq 2$) one has
\begin{align*}
Sym^{p,q}_k(n)=Sym^{p-1,q-1}_{k-2}(n-2)\cdot \Phi(C)
\end{align*}
and the left factor on the right hand side may be assumed to lie in the image of $\Phi$ by induction.\medskip

If $k$ is arbitrary, $q=0$ and $p=k$ one has, modulo $\Phi(C)$,
\begin{align*}
Sym^{k,0}_k(n)	&=\alpha_{k,n}^{p,q}\cdot(x^dh^k+y^kh^k+x^{n-k}y^nh^{2n-k}+x^ny^{n-k}h^{2n-k})z^n\\
				&\equiv[(xh+yh)]^k\cdot[(1+xyh^2)z]^{n-k}\\
				&=\Phi(B)^k\Phi(A)^{n-k}\in\im\Phi.
\end{align*}

Let now $q=0$ and $n>k>p\geq 0$. One has
\begin{align*}
\Phi(L)\cdot Sym_{k-1}^{p,0}(n-2)	&= Sym_k^{p,0}(n)+Sym_{k+2}^{p+2,2}(n)&\\
									&\equiv Sym_k^{p,0}(n)&\text{mod }\Phi(C)
\end{align*}
and so $Sym^{p,0}_k(n)$ is in the image of $\Phi$ by induction.\medskip

Next, let $n>p=k\geq q\geq 1$. One computes
\begin{align*}
\Phi(M)\cdot Sym_{k-1}^{k-1,k-1}(n-1)&=Sym_k^{k,q}(n)+Sym_k^{k-1,q-1}(n)
\end{align*}
and so by induction $Sym^{p,q}_k(n)$ is in the image of $\Phi$ whenever $Sym_k^{k-1,q-1}(n)$ is. But if $q\geq 2$, the latter is a multiple of $\Phi(C)$ and for $q=1$ it is $Sym^{k-1,0}_k(n)$, which is in the image of $\Phi$ by the previous paragraph.\medskip

It remains to treat the case $k=n$ and $p$ or $q$ in $\{0,n\}$. For $q=0$ and $n-1> p\geq 0$, one has the equation
\begin{align*}
\Phi(M)\cdot Sym_{n-1}^{p,0}(n-1)&=Sym_n^{p,0}(n)+Sym_n^{p+1,1}(n)
\end{align*}
and the right summand is a multiple of $\Phi(C)$. For $p=n-1, q=0$, one has directly
\begin{align*}
\Phi(M)\cdot Sym_{n-1}^{n-1,0}(n-1)&=Sym_n^{n-1,0}(n).
\end{align*}
Since for $k=n$, one has the restriction $p+q\leq n$, this ends the verification of surjectivity.

The second part of the theorem now follows, since $\Phi$ vanishes on the generator of $I$ and hence induces in each degree a surjective map
\[
(\Z[A,B,C,M,L]/I)_n\longrightarrow \RB'_n.
\]
But both sides are free $\Z$-modules of the same rank, so the map has to be injective as well.\end{proof}

\begin{rem}\label{rem: generators for RB up to degree 2}
	One may now check that $A,B,C,L$ can indeed be realized by complex manifolds as
\begin{align*} 
A&=rb([\C\Pro^1])\\
B&=rb([E]-[\C\Pro^1])\\
C&=rb([\C\Pro^1\times\C\Pro^1]-[\C\Pro^2])\\
L&=rb([H]-2[\C\Pro^2]+[\C\Pro^1\times\C\Pro^1]),
\end{align*} 
where $E$ denotes an elliptic curve and $H$ a Hopf surface.\footnote{As in Theorem \ref{thm: H=Hform}, instead of $\CP^2$ and $H$, we may also use an arbitrary K\"ahler surface of signature $\pm 1$ and an arbitrary non-K\"ahler surface, although the exact formulas may differ.}
In particular, we recover the Hodge ring of K\"ahler manifolds from  \cite{kotschick_hodge_2013} (c.f. section \ref{sec: Hodge, Betti, Chern}) as $\cH_\ast^K=\Z[A,B,C]\subseteq\RB_\ast$. However, $M$ is \textbf{not} in the image of $rb$, since it violates condition (B4). For instance, it would correspond to a (formal linear combination of) curve(s) not satisfying the $\del\delbar$-lemma. So we see that there is a single generator being responsible for the inclusion $\RB_\ast\subseteq \RB'_\ast$ failing to be an equality.
\end{rem}

It remains possible (and plausible) that the ring $\RB^{form}_\ast$ equals $\RB_\ast$. To show this, one might proceed as before: write down generators for $\RB^{form}_\ast$ and show that they can all be realized by $\Z$-linear combinations of actual compact complex manifolds. Unlike in the previous cases, however, it turns out that $\RB^{form}_\ast$ is not finitely generated:

\begin{thm}
Under the isomorphism
\[
\overline{\Phi}:\Z[A,B,C,L,M]/I\longrightarrow \RB'_\ast,
\]
the subring $\RB^{form}_\ast\subseteq \RB'_\ast$ corresponds to the subring generated by $A$, $B$, $C$, $L$, $ABM$, $CM$ and the collection $AM^{n+1}$, $M^nL$ for $n\in\Z_{\geq 1}$.\medskip

In terms of the polynomials $Sym_k^{p,q}(n)$, a set of generators for $\RB^{form}_\ast$ is given by  
\[
Sym_0^{0,0}(1)=\Phi(A),~ Sym_1^{1,0}(1)=\Phi(B),~Sym_2^{1,1}(2)=\Phi(C),~ Sym_1^{0,0}(2)=\Phi(L),\]
the polynomials
\begin{align*}
Sym^{2,1}_2(3)&= \Phi(A B M)- \Phi(BL)\\
Sym^{1,1}_3(3)&= \Phi(C M)
\end{align*}
and the collection of polynomials
\begin{align*}
L_n&:=Sym_{n-1}^{0,0}(n)\\
M_n&:=Sym_{n-1}^{n-1,n-1}(n)
\end{align*}
for $n>2$.
\end{thm}

\begin{rem}
	The above theorem presents one infinite generating set. A priori it may still be possible to choose a finite collection of generators for $\RB^{form}_\ast$. However, such a collection would contain an element of maximal degree, and one may verify that $L_n$ and $M_n$ do not lie in the ideal generated by elements of degree $<n$.
\end{rem}

\begin{proof}
Recall that $\RB^{form}_\ast\subseteq \RB'_\ast$ is defined by
\[
\RB^{form}_\ast=\left\{\sum b_k^{p,q}(n)\cdot Sym^{p,q}_k(n)\in \RB'_n\mid b^{p,0}_n(n)=b_1^{1,1}(n)=0~\forall p<n\right\}\subseteq \RB'_\ast
\]

and that $\Phi$ gives a surjective map
\[
\Phi: \Z[A,B,C,M,L]\longrightarrow \RB'_\ast.
\]
We are going to compute the inverse image of $\RB^{form}_\ast$ under this map. Denote
\[
S:=\left\{\sum_{a,b,c,d,e\in\Z_{\geq 0}} c_{a,b,c,d,e}A^aB^bC^cM^dL^e\mid c_{a,0,0,1,0}=c_{0,b,0,d,0}=0~\forall a,b,d-1\geq 0\right\}
\]
and define $\overline{S}$ to be the reduction of $S$ modulo $I$. The set $S$ (and a fortiori $\overline{S}$) is a graded subring of $Z[A,B,C,M,L]$ (resp. $\Z[A,B,C,M,L]/I$).\medskip

\textbf{Claim:}
\textit{$\overline{S}$ is generated by $A,B,C,L,ABM,CM$ and the collection $AM^{n+1}, M^nL$ for $n\in\Z_{\geq 1}$.}\medskip

Indeed, the (images of) the monomials $A^aB^bC^dM^dL^e$ which do not violate the defining conditions of $S$ generate $\overline S$. Fix such a monomial. If $d=0$, it is a product of $A,B,C$ and $L$.\medskip

If $d\geq 1$ and $e\geq 1$, it is a product of $A,B,C,L$ and some $M^n L$.\medskip

If $e=0$ and $d=1$, either $c\geq 1$ or $a\neq 0\neq b$. In the first case we have a product of $A,B,C, CM$, in the second case a product of $A,B$ and $ABM$.\medskip

If $e=0$ and $d\geq 2$, either $c=0$, in which case $a\geq 1$ and we have a product of $A,B$ and $AM^d$. Or $c\geq 1$ and, using that we work modulo $I$,
\begin{align*}
A^aB^bC^c M^d\equiv A^aB^b C^{c-1}(AM^{d-1}L-L^2M^{d-2}-CA^2M^{d-2}+4C^2 M^{d-2}).
\end{align*}
The first three summands in the bracket are multiples of the claimed generators, so one may inductively reduce to the case $d\leq 1$, where the claim has been proven.\medskip

\textbf{Claim:} 
$\Phi(S)\subseteq \RB^{form}_\ast$.\medskip

Since $\Phi$ factors through reduction modulo $I$, this may be checked on the generators for $\overline{S}$, the nontrivial cases being $AM^{n+1}$ and $M^nL$. There we have

\begin{align*}
\Phi(AM^n)	&= Sym^{0,0}_0(1)\cdot [Sym_1^{0,0}(1)]^n\\
			&= (1+xyh^2)z\cdot [(h+xyh)z]^n\\
			&= (1+xyh^2)z\cdot \left(\sum_{i=0}^n\binom{n}{i}x^iy^i h^n z^n \right)\\
			&= \sum_{i=0}^n\binom{n}{i} Sym_n^{i,i}(n+1)&\in \RB_{n+1}'
\end{align*}
and
\begin{align*}
\Phi(LM^n)	&= Sym^{0,0}_1(2)\cdot[Sym_1^{0,0}(1)]^n\\
				&= (h+x^2y^2h^3)z^2\cdot[(h+xyh)z]^n\\
				&= (h+x^2y^2h^3)z^2\cdot \left(\sum_{i=0}^n\binom{n}{i}x^iy^i h^n z^n \right)\\
				&= \sum_{i=0}^n\binom{n}{i} Sym_{n+1}^{i,i}(n+2)&\in \RB_{n+2}'.
			\end{align*}

By staring for a moment at the defining conditions for $S_n$ and $\RB^{form}_n$, one arrives at the following\medskip

\textbf{Observation:} For each $n$, the modules $R_n:=\Z[A,B,C,M,L]_n/S_n$ and $\RB'_n/\RB^{form}_n$ are free of the same rank, which equals $n$ for $n=0,1$ and $n+1$ otherwise.\medskip

Since $\Phi$ surjects onto $\RB'_\ast$, the cokernels $R_n\cong \RB'_n/\RB^{form}_n$ are isomorphic. Hence, applying the five-lemma to the map of short exact sequences
	\[
	\begin{tikzcd}
	0\ar[r]&S_n\ar[r]\ar[d]&\Z[A,B,C,M,L]_n\ar[r]\ar[d,]&R_n\ar[r]\ar[d]&0\\
	0\ar[r]&\RB_n^{form}\ar[r]&\RB'_n\ar[r]&\RB'_n/\RB^{form}_n\ar[r]&0,
	\end{tikzcd}
	\]
one obtains surjectivity of $\Phi$ from $S$ to $\RB^{form}_\ast$ and therefore an isomorphism $S/(S\cap I)\cong \RB^{form}_n$.\medskip

To complete the proof, we have to show that instead of $\Phi(AM^n)$ and $\Phi(LM^n)$, we might just as well take $L_n=Sym^{0,0}_{n-1}(n)$ and $M_n=Sym^{n-1,n-1}_{n-1}(n)$ as generators. In fact, by the formula for $\Phi(LM^n)$ above, we have
\[
\Phi(LM^n)-L_{n+2}=\Phi(C)\cdot\sum_{i=1}^n\binom{n}{i} Sym_{n-1}^{i-1,i-1}(n),
\]	
and the all summands in the sum on the right are again in $\RB'_n$, so the claim follows by induction. Similarly, 
\begin{align*}
\Phi(AM^n)-L_{n+1}-M_{n+1}&=\sum_{i=1}^{n-1}\binom n i Sym_n^{i,i}(n+1)\\
							&=\Phi(C)\cdot\sum_{i=1}^{n-1}\binom n i Sym_{n-2}^{i-1,i-1}(n-1).
\end{align*}
 \end{proof}

\begin{rem}\label{rem: generators for RB}
We have already seen in Remark \ref{rem: generators for RB up to degree 2} that $\RB_{\leq 2}=\RB_{\leq 2}^{form}$. We will also exhibit manifolds realizing all generators in degree $3$, all $L_n$, and the $M_n$ for $n$ odd. In particular, this shows that not only $\RB^{form}_\ast$ but also $\RB_\ast$ itself cannot be finitely generated. Since some of this discussion overlaps with construction problems appearing in the next sections, we delay it to Section \ref{sec: constructions}.
\end{rem}

\section{Higher pages of the Fr\"olicher spectral sequence}\label{sec: Frolicher numbers}
Instead of only the Dolbeault cohomology, we can also take into account higher pages of the Fr\"olicher spectral sequence
\[
E_1^{p,q}(X)=H^{p,q}_{\delbar}(X)\Longrightarrow H^{p+q}_{dR}(X).
\]
\begin{definition}
	Let $X$ be a compact complex manifold. The \textbf{$r$-th Fr\"olicher numbers} are the integers 
	\[
	e_r^{p,q}(X):=\dim E_r^{p,q}(X).
	\]
\end{definition}
The K\"unneth-formula for Dolbeault cohomology implies a K\"unneth formula for all higher pages of the spectral sequence (see the proof of Proposition \ref{prop: refined Kuenneth}), namely:

\begin{prop}
	Given a product of compact complex manifolds $X\times Y$, the pullback from the factors induces, for any $p,q,r\in\Z$, natural isomorphisms
	\[
	\bigoplus_{\substack{p_1+p_2=p\\q_1+q_2=q}}E_r^{p_1,q_1}(X)\otimes E_r^{p_2,q_2}(Y)\cong E_r^{p,q}(X\times Y).
	\]
\end{prop}

Hence there is a homomorphism from the graded ring of isomorphism classes of compact complex manifolds
\[
e_r: \CM_\ast\longrightarrow \Z[x,y,z],
\] 
induced by sending an $n$-dimensional manifold $X$ to its $r$-th Fr\"olicher polynomial
\[
X\longmapsto e_r(X):=\sum_{p,q=0}^ne_r^{p,q}(X)x^py^qz^n.
\]
Denote the image of this map by $\cH^r_\ast$. In Section \ref{sec: Hodge, Betti, Chern}, we have considered the case $\cH^1_\ast=\cH_\ast$. Because the higher pages of the Fr\"olicher spectral sequence also satisfy Serre-duality (see the proof of Proposition \ref{prop: obvious linear relations between refined Betti numbers} or \cite{stelzig_structure_2018}, \cite{popovici_higher-page_2020-1}, \cite{milivojevic_another_2020-1}), $\cH^r_\ast$ is contained in $\cH^{form}_\ast=\cH_\ast$. Furthermore, we have seen that $\cH^{form}_\ast$ is generated by manifolds for which the Fr\"olicher spectral sequence degenerates. For such manifolds, the Hodge polynomial coincides with the $r$-th Fr\"olicher polynomial for any $r\geq 1$. Hence, we obtain the following immediate generalisation of Theorems \ref{thmintro: HDRC rels} and \ref{thmintro: bimeromorphic invs}:

\begin{thm}
	For any $r\geq 1$, there is an equality $\cH^r_\ast=\cH_\ast$. In particular, for any fixed $r\geq 1$, there are no universal linear relations between the $r$-th Fr\"olicher numbers of compact complex manifolds other than the ones induced by Serre duality.
\end{thm}
\begin{cor}
	For any $r\geq 1$, the only linear combinations between $r$-th Fr\"olicher numbers which are bimeromorphic invariants are (modulo Serre duality) linear combinations of the $e_r^{p,0}$ and $e_r^{0,q}$ only.
\end{cor}

Of course, this leaves open the question of what the universal linear relations between Fr\"olicher numbers for different $r$ are. For instance, the following ones are well-known:
{\it
\begin{enumerate}
	\item[{\it(F1)}] Connected components: $e_1^{0,0}=e_r^{0,0}$ for all $r$.
	\item[{\it(F2)}] Euler Characteristics:
	Let \[\chi_p^r(X):=\sum_{q\geq 0} (-1)^q e^{p-2(r-1)q, r q}(X).\] Then $\chi_p^r(X)=\chi_p^{s}(X)$ for all $r,s\geq 1$.
	\item[(F3)] No differential ends in degree $(n,0)$, i.e. $e_1^{n,0}(X)=e_r^{n,0}(X)$ for all $r$.
	\item[(F4)] No differential starts in degree $(0,0)$, i.e. $e_1^{0,0}(X)=e_r^{0,0}(X)$ for all $r$.
	\item[(F5)] For $\dim_\C X=n$, the Fr\"olicher spectral sequence degenerates latest at the $n$-th page, i.e. $e_n^{p,q}(X)=e_{n+k}^{p,q}(X)$ for all $k\geq 0$. (This is obtained from the trivial degeneration bound by dimension combined with {\it (F3)}.)
	\item[(F6)] For $\dim_\C X\leq 2$, the Fr\"olicher spectral sequence degenerates at the first page, i.e. $e_1^{p,q}(X)=e_{2}^{p,q}(X)$ (in addition to {\it (F5)}).
\end{enumerate}
}
As before, one obtains a map
\[
E: \CM_\ast\longrightarrow (\Z[x,y])^\N[z]
\]
induced by collecting all Fr\"olicher polynomials:
\[
X\mapsto \left(\sum_{p,q=0}^n e_r^{p,q}(X)x^py^q\right)_{r\geq 1}\!\!\!\!\!\!\!\cdot z^{\dim_\C X}
\]
Denote its image by $\mathcal{FS}_\ast$, the \textbf{Fr\"olicher ring} of compact complex manifolds.
\begin{rem}\label{thm: FS not fin gen}
	The ring $\mathcal{FS}_\ast$ is not finitely generated.
\end{rem}
\begin{proof}
	If there was a finite generating set, say $X_1,...X_k$, then by property $5$ above, all coefficients $(P_1,...,P_r,...)\in(\Z[x,y])^\N$ of elements in $\mathcal{FS}_\ast$ would satisfy $P_r=P_{r'}$ whenever $r,r'\geq N:=\max\{\dim_\C X_i\}$. In other words, there would be a  number $N$ such that the Fr\"olicher spectral sequence of all compact complex manifolds, regardless of their dimension, degenerates at stage $N$. This is known to be false, see \cite{bigalke_erratum_2014}.
\end{proof}
The conditions {\it (F1) -- (F6)}, together with Serre duality on each page, describe a natural candidate $\mathcal{FS}^{form}_\ast$ for the image. One could now proceed, as in the previous section, to compute generators for $\mathcal{FS}^{form}_\ast$ and try to realize them by compact complex manifolds. We do not spell this out for two reasons: On the one hand, the conditions defining $\FS_\ast^{form}$ are a bit technical and we believe the description of the universal ring of cohomological invariants in the next section is the better framework to incorporate the Fr\"olicher numbers. On the other hand, we are quite far at the moment from actually realizing the generators by manifolds, and, without that step, the formal work does not seem to be rewarding. See Section \ref{sec: constr. probs} for a summary of the open construction problems.

\section{The universal ring of cohomological invariants}\label{sec: univ ring}
In this section we connect the previous results with the theory developed in \cite{stelzig_structure_2018} and define a ring encoding `all' cohomological invariants. The following structure theorem will underlie much of the discussion:

\begin{thm}[\cite{khovanov_faithful_2020},\cite{stelzig_structure_2018}]\label{thm: structure dc}$\,$
	\begin{enumerate}
	\item Any bounded double complex of $\C$-vector spaces can be decomposed into a direct sum of indecomposable double complexes. The multiplicity with which an (isomorphism class of an) indecomposable double complex occurs is the same for any such decomposition.
	\item Every indecomposable double complex is isomorphic to one of the following:
	\begin{enumerate}
		\item squares
			\[\begin{tikzcd}
			\C\ar[r,"\sim"]&\C\\
			\C\ar[u,"\sim"]\ar[r,"\sim"]&\C\ar[u,"\sim"]
			\end{tikzcd}\]
	
		\item even length zigzags
		\[
		\begin{tikzcd}
			\C\ar[r,"\sim"]&\C
		\end{tikzcd}\quad,\quad  
		\begin{tikzcd}
		\C\\
		\C\ar[u,"\sim"]
		\end{tikzcd}\quad,\quad 
		\begin{tikzcd}
		\C \ar[r,"\sim"] & \C \\
		& \C \ar[u,"\sim"] \ar[r,"\sim"] & \C
		\end{tikzcd}\quad ,\quad  ...
		\]
		\item odd length zigzags
		\[
		\C\quad ,\quad 
		\begin{tikzcd}
		\C\ar[r,"\sim"]&\C\\
		&\C\ar[u,"\sim"]
		\end{tikzcd}\quad , \quad 
		\begin{tikzcd}
		\C&\\
		\C\ar[r,"\sim"]\ar[u,"\sim"]&\C
		\end{tikzcd}
		\quad , \quad ...
		\]
\end{enumerate}
	\item For $X$ a compact complex manifold and $A_X$ its double complex of $\C$-valued forms, in any decomposition $A_X=A_X^{sq}\oplus A_X^{zig}$ into a summand consisting of squares and one consisting of zigzags, the latter is finite dimensional.
	\end{enumerate}
\end{thm}

Let us call complexes satisfying the property in point $3$ of Theorem \ref{thm: structure dc} \textbf{essentially finite dimensional}. Denote by $\cD$ the category of essentially finite dimensional bounded double complexes of $\C$-vector spaces with linear maps compatible with bigrading and differentials. Let $\cC$ denote the category of compact complex manifolds. Sending a manifold $X$ to the double complex $A_X$ and a holomorphic map to the pullback of forms yields a functor $A_{\_}: \cC\rightarrow \cD$.
\begin{definition}[c.f. \cite{stelzig_structure_2018}]\label{def: cohomological functors}
	Let $\V$ denote the category of finite dimensional $\C$-vector spaces. 
	\begin{enumerate}
		\item A linear functor $H: \cD\longrightarrow \V$ is called \textbf{cohomological} if it sends direct sums of squares to zero.
		\item A functor $H:\cC\rightarrow \V$ is called \textbf{cohomological} if it factors as $H=H'\circ A_{\_}$ for a cohomological functor $H':\cD\longrightarrow \V$.
		\item For any cohomological functor $H$ on $\cC$ (resp. $\cD$), the composition $\dim \circ H$, considered as a map from isomorphism classes of objects to $\N$, is called a \textbf{cohomological invariant} of compact complex manifolds (resp. of double complexes).
	\end{enumerate}
\end{definition}

\begin{ex}
As mentioned in the introduction, the functors $H_{dR}^k$, $H_{\delbar}^{p,q}$, $H_{\del}^{p,q}$, $H_{BC}^{p,q}$, $H_{A}^{p,q}$, as well as their higher page analogues $E_r^{p,q}$, $\bar{E}_r^{p,q}$, $E_{r,BC}^{p,q}$, $E_{r,A}^{p,q}$, the cohomologies of the Schweitzer complexes, the Varouchas groups and the $H_k^{p,q}$ are all cohomological, so their dimensions yield cohomological invariants. (See \cite{popovici_higher-page_2020}, \cite{stelzig_remarks_2022} for $E_{r,BC}^{p,q}$, $E_{r,A}^{p,q}$ and the Schweitzer complexes. The other cases are easy to check).
\end{ex}

\begin{rem}
	The reader preferring a less phenomenological justification for the name cohomological functor may find comfort in the fact that direct sums of squares are exactly the projective (and injective) objects in $\cD$ (see \cite{khovanov_faithful_2020}).
\end{rem}

\begin{rem}
Because linear functors commute with direct sums, any cohomological invariant $h$ of double complexes can be written as a finite sum of multiplicities of zigzags, i.e. 
\[
h(\_)=\sum_{Z\text{ zigzag}}a_{h,Z}\cdot\mult_Z(\_ )\quad\text{ for some }a_{h,Z}\in\N.\]
\end{rem}

The following Lemma, expanding on the previous Remark, was essentially shown in \cite{stelzig_structure_2018}:

\begin{lem} \label{lem: same cohomological invariants}
	Let $X$ and $Y$ be $n$-dimensional compact complex manifolds. The following conditions are equivalent:
	\begin{enumerate}
		\item $h(X)=h(Y)$ for all cohomological invariants $h$.
		\item $e_r^{p,q}(X)=e_r^{p,q}(Y)$ and $b_k^{p,q}(X)=b_k^{p,q}(Y)$ for all $k,r,p,q\in\Z$.
		\item Every isomorphism type of zigzags occurs with the same multiplicity in any decomposition into indecomposables of $A_X$ and $A_Y$.
		\item $A_X$ and $A_Y$ are $E_1$-isomorphic, i.e. there exists a map of double complexes $f:A_X\rightarrow A_Y$ s.t. $H_{\delbar}(f)$ and $H_{\del}(f)$ are isomorphisms. 
	\end{enumerate}
\end{lem}

\begin{proof}
	Clearly, $1\Rightarrow 2$. By \cite[sect. 2]{stelzig_structure_2018}, the multiplicities of even zigzags are counted by differentials on the Fr\"olicher spectral sequence and the multiplicities of odd length zigzags are counted by the refined Betti numbers. In particular condition $3$ is equivalent to $b_k^{p,q}(X)=b_k^{p,q}(Y)$ and $\dim \im d_r^{p,q}\cap E_r(X)=\dim\im d_r^{p,q}\cap E_r(Y)$ for all $r,p,q\in\Z$, where $d_r^{p,q}$ denotes the differential on the $r$-th page of the Fr\"olicher spectral sequence. Since the complexes appearing on each page of the Fr\"olicher spectral sequence are bounded, and each page can be computed as the cohomology of the previous one, the dimensions of the images of the differentials are determined by the dimensions $e_r^{p,q}$ for all $r,p,q\in\Z$, hence $2\Rightarrow 3$. Since a linear functor commutes with direct sums, it is determined by its values on indecomposable complexes. Thus $3\Rightarrow 1$. The equivalence between $3$ and $4$ is \cite[Prop. 12]{stelzig_structure_2018}. 
\end{proof}

Let $\cI\subseteq \CM_\ast$ be the ideal generated by differences $[X]-[Y]$ of equidimensional manifolds such that $h(X)=h(Y)$ for all cohomological invariants $h$.

\begin{definition}\label{def: univ. ring}
	The quotient $\cU_\ast:=\CM_\ast/\cI$ will be called the \textbf{universal ring of cohomological invariants}.
\end{definition}

By construction, the universal ring of cohomological invariants $\cU_\ast$ has natural surjective maps to the rings $\HDR_\ast$, $\RB_\ast$, and $\FS_\ast$ introduced in sections \ref{sec: Hodge, Betti, Chern}, \ref{sec: refined de Rham ring}, \ref{sec: Frolicher numbers}. 
Since the latter two rings are not finitely generated, we further get:
\begin{cor}
	$\cU_\ast$ cannot be finitely generated.
\end{cor}

By point $2$ of Lemma \ref{lem: same cohomological invariants}, we obtain:
\begin{cor}
The universal ring of cohomological invariants $\cU_\ast$ is isomorphic to the image of the diagonal map into the product of Fr\"olicher and refined de Rham ring
\[
\CM_\ast\longrightarrow \mathcal{FS}_\ast\times\RB_\ast.
\]
\end{cor}

We will now define a `diagrammatic' ring which contains $\cU_\ast$ (and may be equal to it): Denote by $\cR_0$ the free abelian group on the set of isomorphism classes of finite dimensional bounded double complexes modulo the relations $[A\oplus B]=[A]+[B]$. $\cR_0$ has a $\Z$-basis consisting of all isomorphism classes of indecomposable double complexes. By setting $[A]\cdot [B]:=[A\otimes B]$, the group $\cR_0$ is equipped with a ring structure. The free abelian group on all squares, $I_{Sq}\subseteq \cR_0$ is an ideal and we set $\cR:=\cR_1:=\cR_0/I_{Sq}$. In other words, $\cR$ is the Grothendieck ring of the category of essentially finite dimensional bounded double complexes localized at $E_1$-isomorphisms. Additively, $\cR$ is the free abelian group on isomorphism classes of zigzags. The multiplicative structure is described explicitly in \cite[sect. 3]{stelzig_structure_2018}. The following is a well-defined map of rings (\cite[sect. 4]{stelzig_structure_2018}) and by Lemma \ref{lem: same cohomological invariants}, $\cU_\ast$ may be identified with the image:
\begin{align*}
\CM_\ast&\longrightarrow \cR[z]\\
[X]&\longmapsto [X]:=[A_X]\cdot z^{\dim_\C X}=\sum_{Z\text{ zigzag}} \mult_Z(A_X)[Z] z^{\dim_\C X}.
\end{align*}

The isomorphism class of any indecomposable complex $Z$ is uniquely determined by its support $\supp(Z)=\{(p,q)\mid Z^{p,q}\neq 0\}$. For any isomorphism class of an indecomposable complex $Z$, there is the conjugate isomorphism class $\tau Z$ with support $\supp (\tau Z)=\{(p,q)\mid (q,p)\in\supp(Z)\}$ and, fixing a degree $n$, the dual isomorphism class $\sigma Z$ with support $\supp(\sigma Z)=\{(p,q)\mid (n-p,n-q)\in Z\}$. 

\begin{definition}\label{def: univ form}
Let $\cR$ be as above. We denote by $\cU_\ast^{form}\subseteq \cR[z]$ the graded subring consisting in degree $n$ of those elements $\sum_Z m_Z[Z]z^n$ which satisfy:	
{\it
	\begin{enumerate}
			\item[(U1)] Bounded support: The support of $Z$ is contained in $\{0,...,n\}^2$.
			\item[(U2)] Serre and conjugation symmetry: $m_Z=m_{\sigma Z}=m_{\tau Z}=m_{\sigma\tau Z}$.
			\item[(U3)] Only dots in the corners: \\If $\{P\}\subsetneq supp(Z)$ for $P=(0,0)$ or $(n,0)$, then $m_Z=0$.
			\item[(U4)] Degenerate FSS: $m_Z=0$ for even length zigzags if $n\leq 2$.
		\end{enumerate}
		}
\end{definition} 
We have $\cU_\ast\subseteq\cU_\ast^{form}$ (c.f. \cite[sect. 4]{stelzig_structure_2018}). By construction $\cU_n^{form}$ is the free abelian group on certain elements denoted $Sym_Z(n)$ for a zigzag $Z$, defined to be $z^n$ times the sum over all elements in the orbit $\Z/2\Z\times\Z/2\Z$-action on isomorphism classes of zigzags given by $\sigma$ and $\tau$.  
\begin{ex}\label{ex: symmetrized zigzag}
Consider the zigzag
\[
L:= [\C^{1,0}\overset{\Id}{\longrightarrow}\C^{2,0}],
\]
where the superscripts indicate the bidegree in which the copy of $\C$ sits. Then 
\[
Sym_L(3)=[\C^{1,0}\overset{\Id}{\longrightarrow}\C^{2,0}]\oplus [\C^{0,1}\overset{\Id}{\longrightarrow}\C^{0,2}]\oplus [\C^{1,3}\overset{\Id}{\longrightarrow}\C^{2,3}]\oplus[\C^{3,1}\overset{\Id}{\longrightarrow}\C^{3,2}].
\]
\end{ex}

We will often depict the $Sym_Z(n)$ in a diagrammatic way, using a checkerboard of size $n\times n$ (which makes the degree-tracking $z^n$-factor redundant). For instance, if $L$ is as in Example \ref{ex: symmetrized zigzag}, we have
\[
Sym_{L}(3)=\begin{gathered}\includegraphics[scale=1]{S1110_3}\end{gathered},
\]
and a basis for the first two degrees of $\cU_\ast^{form}$ is given as follows:
 \[
 \cU_1^{form}=\Z\cdot \begin{gathered}\includegraphics[scale=1]{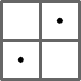}\end{gathered} \oplus \Z\cdot \begin{gathered}\includegraphics[scale=1]{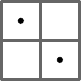}\end{gathered},
 \]
 \[
 \cU_2^{form}=\Z\cdot \begin{gathered}\includegraphics[scale=1]{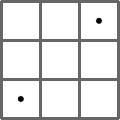}\end{gathered}\oplus
 \Z\cdot\begin{gathered}\includegraphics[scale=1]{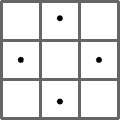}\end{gathered}\oplus \Z\cdot\begin{gathered}\includegraphics[scale=1]{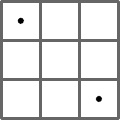}\end{gathered}\oplus
 \Z\cdot\begin{gathered}\includegraphics[scale=1]{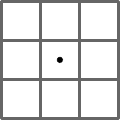}\end{gathered}\oplus
 \Z\cdot\begin{gathered}\includegraphics[scale=1]{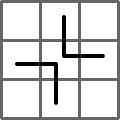}\end{gathered}.
 \]
 Finally, it will sometimes be convenient to also have a non-diagrammatic way of writing down a specific $Sym_Z(n)$. This amounts to defining labels for all possible (isomorphism classes of) zigzags. One such labeling was given in \cite{stelzig_structure_2018} which we recall here briefly: The isomorphism class of an odd zigzag $Z$ is labeled as $S_k^{p,q}$ if $b_k^{p,q}$ is the unique nonzero refined Betti number of $Z$. Similarly, if $Z$ is an even zigzag s.t. there is a nonzero differential $d:E_r^{p,q}(Z)\to E_r^{p+r,q-r+1}(Z)$ in the `Fr\"olicher' (i.e. column) spectral sequence, we denote its isomorphism class by $S_{1,r}^{p,q}$.\footnote{The first $1$ in the subscript stands for the use of the column spectral sequence as opposed to the row spectral sequence. Even zigzags not contributing to the column spectral sequence contribute to the row spectral sequence and are labeled as $S_{2,r}^{p,q}$.} E.g., $S^{1,0}_{1,1}$ equals $L$ from Example \ref{ex: symmetrized zigzag} and
 \[
 Sym_{S_{1,1}^{1,0}}(3)=\begin{gathered}\includegraphics[scale=1]{S1110_3}\end{gathered}\quad  Sym_{S_2^{2,2}}(3)=\begin{gathered}\includegraphics[scale=1]{S222_3}\end{gathered}.
 \]
 
\begin{rem} There is a section to the map $\cU_\ast^{form}\to\RB_\ast^{form}$ given by
 \begin{align*}
 s:\RB_\ast^{form}&\longrightarrow \cU_\ast^{form}\\
 Sym_k^{p,q}(n)&\longmapsto Sym_{S_k^{p,q}}(n),
 \end{align*}
 allowing us to identify $\RB_\ast^{form}$ with a subring of $\cU^{form}_\ast$. Note however that $s\circ rb(X)=[X]\in \cU_\ast$ only for manifolds whose Fr\"olicher spectral sequence degenerates at the first page. We encourage the reader to revisit Section \ref{sec: refined de Rham ring} from this diagrammatic point of view.
\end{rem}
 
 \begin{lem}
 	The map	
 	\begin{align*}
 	\cH_\ast^K&\longrightarrow\cU^{form}_\ast\\
 	x^py^q z^n&\mapsto [S_{p+q}^{p,q}]z^n
 	\end{align*}
 	identifies the Hodge ring of K\"ahler manifolds introduced by Kotschick and Schreieder (c.f. also section \ref{sec: Hodge, Betti, Chern}) with the subring of $\cU^{form}_\ast$ consisting of linear combinations of zigzags of length one (`dots').
 \end{lem}
 
 Under this identification, we have
 
 \begin{align*}
 A=[\CP^1]=\img{S000_1},&& B=[E]-[\CP^1]=\img{S110_1} 
 \end{align*}
 and 
 \begin{align*}
 C=[(\CP^1)^2]-[\CP^2]&=\img{S211_2}.
 \end{align*}
 
 As before, we obtain a ring-theoretic reformulation of the question whether all linear relations between arbitrary cohomological invariants reduce to the ones listed above:
 
\begin{qu}
	Is $\cU_\ast=\cU_{\ast}^{form}$?
\end{qu}

It will be often convenient to calculate modulo the Hodge ring of Kähler manifolds $\cH^K_\ast$. E.g. for $S_{NK}$ any non-K\"ahler surface, we have
\begin{align*}
[S_{NK}]& \equiv \begin{gathered}\includegraphics[scale=1]{S100_2}\end{gathered} \in \cU_2/\cH_2^K,
\end{align*}

which is a generator for $\cU_2^{form}/\cH^K_2$.
\begin{prop}\label{prop: universal ring in small degrees}$\,$
	\begin{enumerate}
		\item There are equalities $\RB_{\leq 2}=\cU_{\leq 2}=\cU_{\leq 2}^{form}$
		\item The quotient $\cU_3^{form}/\cU_3$ is either trivial or generated by 
		\[
		Sym_{S_{1,2}^{1,0}}(3)=\img{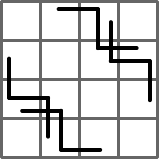}.
		\]
	\end{enumerate}
\end{prop}
\begin{proof}
The first statement follows from the preceding discussion.
 For the second point, we write out a basis of $\cU^{form}_3/\cH^K_\ast$:

\begin{align*}
\frac{\cU_3^{form}}{\cH^K_3}=&\Z\cdot\img{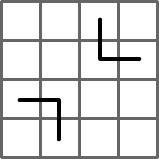} \oplus \Z\cdot \img{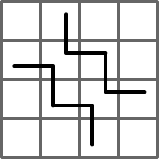}\oplus\Z\cdot \img{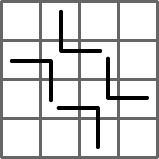} \oplus \Z\cdot\img{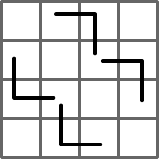} ~\oplus\\
& \Z\cdot\img{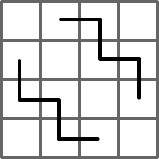} \oplus\Z\cdot\img{S311_3}\oplus \Z\cdot\img{S1110_3}\oplus\Z\cdot\img{S1101_3}~ \oplus\\
&\Z\cdot\img{S1111_3}\oplus\Z\cdot\img{S1102_3} \oplus \Z\cdot\img{S1201_3}\oplus\Z\cdot\img{S1202_3}
\end{align*}
We will see in the next section that all basis elements except possibly $Sym_{S_{1,2}^{1,0}}(3)$ can be realized by formal $\Z$-linear combinations of compact complex manifolds. For now, we note only that 
\[
\img{S110_1}~\cdot~\img{S100_2}=\img{S210_3}\text{ and }\img{S000_1}~\cdot~ \img{S100_2}=\img{S100_3}~+~\img{S311_3},
\]
so by part $1$, we have $Sym_{S_{2}^{1,0}}(3)\in\cU^{form}_3$ and it suffices to realize either $Sym_{S_1^{0,0}}(3)$ or $Sym_{S_3^{1,1}}(3)$.\end{proof}

\section{Construction of certain generators}\label{sec: constructions}
We construct all remaining generators announced in Proposition \ref{prop: universal ring in small degrees} and Remark \ref{rem: generators for RB}. In particular, this proves Theorem \ref{corintro: no unexp. rels coh. invs. ref. betttis}, by the same argument as in Remark \ref{rem: Thm A for Hodge} (with the expected universal relations given in Prop. \ref{prop: obvious linear relations between refined Betti numbers} and Definition \ref{def: univ form}).

\subsection{The Iwasawa manifold and its deformations}
The Iwasawa manifold $\I$ is a complex nilmanifold with left invariant structure determined by the structure equations
\[
d\varphi_1=d\varphi_2=0\text{ and }d\varphi^3=\varphi^1\wedge\varphi^2.
\]
It has a six-dimensional locally complete family of small deformations, which has been stratified by Nakamura into three classes $i,ii,iii$, according to the values of the Hodge numbers \cite{nakamura_complex_1975}. Later, D. Angella has further refined this stratification into five classes $i$, $ii.a$, $ii.b$, $iii.a$ and $iii.b$, taking into account the values of Bott-Chern cohomology as well \cite{angella_cohomologies_2013-1}. The values of Dolbeault and Bott-Chern cohomology on each stratum is (by definition) constant and listed in \cite{angella_cohomologies_2013-1}. Since the dimension is small, using the formulae
\[
2(h_{BC}(X)-h_{\delbar}(X))=\sum_{\substack{Z \text{ zigzag}\\l(Z)\geq 2}}(l(Z)-2)\mult_Z(A_X)
\]

\[
2(h_\delbar(X)-b(X))=\sum_{Z\text{ even zigzag}}\mult_Z(A_X)
\]
(where $l(Z)$ denotes the length of a zigzag) and some elementary combinatorics in the style of \cite[Lemma 4.5]{popovici_higher-page_2020-1}, one verifies that the $E_1$-isomorphism type, i.e. the class in $\cU_\ast$ is also constant along each stratum. We give here the result only,\footnote{See also \cite{flavi_frolicher_2019} where the same result is obtained for examples in all classes except $ii.b$, which is not treated.} for brevity ignoring summands consisting of dots, i.e. the formulae are to be read in $\cU_3/\cH_3^K$:
\begin{align*}
[\I_{i}]&\equiv \begin{gathered}\includegraphics[scale=1]{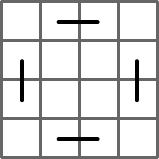}\end{gathered}+ 2\cdot \begin{gathered}\includegraphics[scale=1]{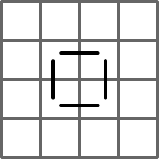}\end{gathered}\\
[\I_{ii.a}]&\equiv \begin{gathered}\includegraphics[scale=1]{S222_3.pdf}\end{gathered}+ \begin{gathered}\includegraphics[scale=1]{S210_3.pdf}\end{gathered} + \begin{gathered}\includegraphics[scale=1]{S1111_3.pdf}\end{gathered}\\
[\I_{ii.b}]&\equiv \begin{gathered}\includegraphics[scale=1]{S221_3.pdf}\end{gathered} + \begin{gathered}\includegraphics[scale=1]{S210_3.pdf}\end{gathered} + \begin{gathered}\includegraphics[scale=1]{S1111_3.pdf}\end{gathered}\\
[\I_{iii.a}]&\equiv \begin{gathered}\includegraphics[scale=1]{S222_3.pdf}\end{gathered}+ 2\cdot\begin{gathered}\includegraphics[scale=1]{S210_3.pdf}\end{gathered}\\
[\I_{iii.b}]&\equiv \begin{gathered}\includegraphics[scale=1]{S221_3.pdf}\end{gathered}+ 2\cdot \begin{gathered}\includegraphics[scale=1]{S210_3.pdf}\end{gathered}
\end{align*}

Since we already know that $Sym_{S_{2}^{1,0}}(3)\in B\cdot \cU_2^{form}\subseteq\cU^{form}_3$, we obtain 
\[
Sym_{S_2^{2,1}}(3),Sym_{S_2^{1,1}}(3),Sym_{S_{1,1}^{1,1}}(3), Sym_{S_{1,1}^{1,0}}(3)\in \cU_3.
\] 
\begin{ex}
	Let us give, as an example, the calculation for the first stratum. Here, by \cite{angella_cohomologies_2013-1}, $h_{BC}(\I_{i})=h_{\delbar}(\I_i)$. Hence all zigzags with nonzero multiplicity have at most length $2$, i.e. the de Rham cohomology is pure and the Fr\"olicher spectral sequence degenerates latest at page $2$. On the other hand $h_\delbar(\I_{i})-b(\I_{i})=12$. Thus, there have to be $12$ length-two zigzags. Since $b_1=4$ and $h_{\delbar}^{1,0}(\I_{i})=3>2=h_{\delbar}^{0,1}(\I_{i})$, one has to have exactly one length two zigzag leaving degree $(1,0)$ horizontally and none leaving degree $(0,1)$ horizontally. Similarly, $b_2=8$ and $h_{\delbar}^{2,0}(\I_{i})=3$, $h_{\delbar}^{1,1}(\I_{i})=6$ and $h_{\delbar}^{2,2}(\I_{i})=2$. Comparing dimensions, we see that there need to be two zigzags leaving degree $(1,1)$ horizontally. By Serre and conjugation Symmetry, all other multiplicites are determined. We note that for this particular stratum, we could also invoke \cite[Cor. 2.7]{popovici_deformations_2022}.
\end{ex}
\subsection{A threefold with $b_3^{1,1}(X)=1$}
Denote by $N$ a complex $3$-dimensional nilmanifold with holomorphic differentials $\varphi_1,\varphi_2,\varphi_3$ and nonzero differential $\varphi_3\mapsto \varphi_1\wedge\varphi_2-\varphi_1\wedge\overline{\varphi}_2$. A tedious but elementary computation (or use of SageMath) yields:
\begin{align*}
[N]	=&~\img{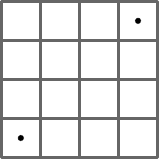}+2\cdot \img{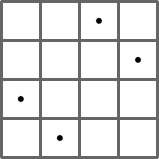}+\img{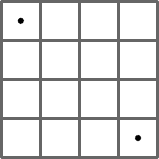}+\img{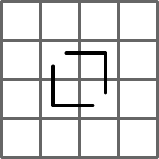}+4\cdot \img{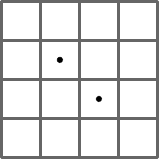}\\
	&~\img{S210_3.pdf}+3\cdot \img{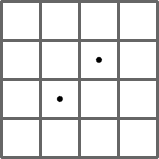}+\img{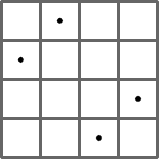}+\img{S221_3.pdf}\\
	\equiv&~ \img{S210_3.pdf}+\img{S221_3.pdf}+\img{S311_3.pdf}	\in\cU_3/\cH_3^K.
\end{align*}

Since the first two summands $Sym_2^{1,0}(3)$ and $Sym_2^{2,1}(3)$ have already been shown to lie in $\RB_\ast\subseteq\cU_\ast$, we get that also $Sym_3^{1,1}(3)\in\RB_3$, resp. $Sym_{S_3^{1,1}}(3)\in\cU_3$.

\subsection{Oeljeklaus-Toma manifolds and the generators $L_n$}
Let $K$ be a number field with exactly $s+2t$ distinct embeddings $\sigma_1,...\sigma_n$ into $\C$, the first $s$ of which have image contained in the reals, while $\sigma_{s+i}=\bar\sigma_{s+t+i}$ are pairs of distinct conjugate complex embeddings. Let $\Oh_K\subseteq K$ denote the ring of algebraic integers of $K$ and $\Oh_K^{\ast, +}\subseteq \Oh_K$ the subgroup of those elements which are positive under all real embeddings. Denoting by $\HH\subseteq \C$ the upper half-plane, there is an action of $\Oh_K\rtimes \Oh_K^{\ast,+}$ on $\HH^s\times\C^t$ by translations and dilations, i.e.	
\[
(a,b).(z_1,...,z_{s+t})=(\sigma_1(b)\cdot z_1 + \sigma_{1}(a),..., \sigma_{s+t}(b)\cdot z_{s+t} + \sigma_{s+t}(a)).
\]
As a consequence of Dirichlet's unit theorem, one can show that there always exist subgroups $U\subseteq \Oh_K^{\ast,+}$ such that the quotient by the above action restricted to $\Oh_K\rtimes U$ is a compact manifold (of dimension $s+t$). We will be only interested in the case $t=1$, where one may take $U=\Oh_K^{\ast,+}$.

\begin{definition}[\cite{oeljeklaus_non-kahler_2005}] An Oeljeklaus-Toma manifold (or OT-manifold) of type $(s,t)$ is any manifold $X(K,U)$ associated to a pair $(K,U)$ as above.
\end{definition}

Combining \cite[Cor. 4.7., Rem. 4.8.]{otiman_hodge_2018} and \cite[Prop. 6.4.]{istrati_rham_2017}, one obtains:

\begin{thm}\label{thm: OT Dolbeault coh}
	Let $n\in\Z_{> 1}$. For an OT-manifold $X$ of type $(n-1,1)$, one has
	\[
	h^{p,q}(X)=\begin{cases} {n-1\choose q} &\text{ if } p=0\\
						{n-1\choose q-1} &\text{ if } p=n\\
						0&\text{ else.}
						\end{cases}
	\]
\end{thm}
\begin{cor}\label{cor: refined Poincare OT}
	Let $n\in\Z_{\geq 1}$ and $X$ an OT-manifold of type $(n-1,1)$. The refined Poincar\'e polynomial of $X$ is given as:
	\[
	rb(X)=\sum_{k=0}^{n-1}\binom{n-1}{k}Sym_{k}^{0,0}(n).
	\]
\end{cor}

\begin{proof}
	From Theorem \ref{thm: OT Dolbeault coh}, we see that $E_\infty^{p,q}(X)=0$ unless $p=0,n$. In particular, the projection
	\[
	H_{dR}^k(X)\to \gr_F^0 H_{dR}^k(X)=E_\infty^{0,k}
	\]
	is an isomorphism for all $k\leq n$, i.e. $F^0H_{dR}(X)=H_{dR}^k(X)$ and $F^1H_{dR}^k(X)=0$ and by conjugation also $\Fbar^0H_{dR}(X)=H_{dR}^k(X)$ and $\Fbar^1 H_{dR}^k(X)=0$. The result now follows from the definition of the $b_k^{p,q}(X)$.
\end{proof}
\begin{rem}
Since the Fr\"olicher spectral sequence for OT-manifolds degenerates, we could also make this calculation in $\cU_\ast$.
\end{rem}
\begin{cor}
	If $\RB_{<n}=\RB_{<n}^{form}$, then $L_n\in\RB_n$.
\end{cor}
\begin{proof}
	By \cite{oeljeklaus_non-kahler_2005}, OT-manifolds of given type $(s,t)\in\Z_{\geq 1}^2$ always exist, so we may take an OT-manifold $X$ of type $(n-1,1)$. Then by Corollary \ref{cor: refined Poincare OT}, we have
		\begin{align*}
		rb(X)&=L_n + \sum_{k=0}^{n-2}\binom{n-1}{k}Sym_k^{0,0}(n).
		\end{align*}
	Since $\RB^{form}_n=(W_{n-1}\RB^{form}_\ast)_n\oplus\langle M_n,L_n\rangle$, where $W_k\RB^{form}_\ast$ is the subring generated in degrees $\leq k$ and no $M_n$-summand appears in $rb(X)$, we conclude that $L_n\in \RB_n$.
\end{proof}

\begin{rem}
In particular, this finishes the proof of $\RB_{\leq 3}=\RB_{\leq 3}^{form}$.
\end{rem}
\subsection{The generators $M_n$ for $n$ odd}\label{sec: Mn}
To construct the generators $M_n$, we search for a complex $n$-fold $X_n$ which satisfies $b_{n-1}^{n-1,n-1}(X_n)=1$. Then, modulo the degree $n$-part of the ring generated by all other generators of degree $\leq n$, we have $rb(X_n)\equiv M_n$. By Lemma \ref{lem: spaces FcapFbar}, $b_{n-1}^{n-1,n-1}(X_n)$ is the dimension of the subspace of $H_{dR}^{n-1}(X_n)$ consisting of classes which admit both a holomorphic and an antiholomorphic representative.\medskip

We have seen that for $n=3$ such a manifold may be found among small deformations of the Iwasawa manifold. The following example, due to Daniele Angella, generalizes this and shows that such $X_n$ exist in every odd dimension.\medskip

Consider the $n=2m+1$-dimensional nilpotent Lie algebra $\fm$ with complex structure defined by the structure equations: 
\begin{align*}
d\varphi^\ell&=0 \qquad \text{ for } \ell\in\{1,\ldots,2m\},\\
 d\varphi^{2m+1}&= -\sum_{\ell=1}^{m} \varphi^{2\ell-1}\wedge \varphi^{2\ell}+\sum_{\ell=1}^{m}\sqrt{-1} \varphi^{2\ell-1}\wedge\bar\varphi^{2\ell-1}.
\end{align*} 
For $i\leq j$ write $\varphi^{\hat{i},\hat{j}}$ for the product of all $\varphi^{l}$ (in ascending order) except $\varphi^i$ and $\varphi^j$. These forms form a basis for $\Lambda_{\fm}^{2m-1,0}$. We have

\[
\del\varphi^{\hat{i},\hat{j}}=\begin{cases}
									-\varphi^1\wedge...\wedge\varphi^{2m} &\text{if }i=2\ell-1,j=2\ell\text{ for some }\ell\in\{1,...,m\}\\
									0 &\text{else},
								\end{cases}
\]
in particular, $\dim\im\del\cap \Lambda_{\fm}^{2m,0}=1$. Since a holomorphic form cohomologous to any form of different bidegree would have to be $\del$-exact, we have $b_{n-1}^{n-1,n-1}(\Lambda_{\fm})\leq 1$. On the other hand, for any $\ell\in\{1,\ldots,m\}$,
\begin{align*}
\left[ \sum_{\ell=1}^{m} \varphi^{2\ell-1} \wedge \varphi^{2\ell} \right] &=\left[ \sum_{\ell=1}^{m}\sqrt{-1} \varphi^{2\ell-1}\wedge\bar\varphi^{2\ell-1} \right] \\
&= \left[ \sum_{\ell=1}^{m} \bar\varphi^{2\ell-1} \wedge \bar\varphi^{2\ell}\right] \in H^2_{dR}(\Lambda_{\fm},\mathbb C) ,
\end{align*}
whence also
\[\left[ \left(\sum_{\ell=1}^{m} \varphi^{2\ell-1} \wedge \varphi^{2\ell}\right)^{m} \right] = \left[ \left(\sum_{\ell=1}^{m} \bar\varphi^{2\ell-1} \wedge \bar\varphi^{2\ell}\right)^{m} \right] \in H^{2m}_{dR}(\Lambda_{\fm},\mathbb C).
\]

The class above is represented by $m!\cdot \varphi^1\wedge\cdots\wedge\varphi^{2m}\neq 0$, so $b_{n-1}^{n-1,n-1}(\Lambda_{\fm})=1$. By Corollary \ref{cor: left invariant cohomoogy invariants}, we may now pick $X_n=\Gamma\backslash N$, where $N$ is the simply connected Lie group corresponding to $\fm$ and $\Gamma$ is any lattice such that the Dolbeault cohomology is computed by left invariant forms.

\begin{rem}
The manifolds $X_n$ constructed above may be taken to be (small) deformations of the complex manifolds $\eta\beta_{2m+1}$ considered by Alessandrini and Bassanelli \cite{alessandrini_compact_1991}. The latter are defined as quotients of the nilpotent Lie group
$$ \left\{
\left(\begin{array}{ccccc|c}
1 & z_1 & z_3 & \cdots & z_{2m-1} & z_{2m+1} \\
\hline
 & 1 & 0 & \cdots & 0 & z_{2} \\
 & & \ddots & \ddots & \vdots & \vdots \\
 & & & 1 & 0 & z_{2m-2} \\
 & & & & 1 & z_{2m} \\
 & & & & & 1
\end{array}\right) 
\in \mathrm{GL}(m+2;\mathbb C) \right\} $$
by the co-compact discrete subgroup generated by such matrices with entries in $\mathbb Z[\sqrt{-1}]$. For $m=1$, this yields the Iwasawa manifold.
A basis for the left-invariant $(1,0)$-forms is given by
\[\omega^1=dz_1,\ldots,\omega^{2m}=dz_{2m}, \quad \omega^{2m+1}=dz_{2m+1}-\sum_{\ell=1}^{m}z_{2\ell-1}dz_{2\ell} ,\]
with structure equations
\[d\omega^1=\ldots=d\omega^{2m}=0, \quad d\omega^{2m+1}=-\sum_{\ell=1}^{m} \omega^{2\ell-1}\wedge\omega^{2\ell}.\]
We consider the small deformation corresponding to the direction
\[
s(t)=\left[ \sqrt{-1}t\sum_{\ell=1}^{m} \left(\frac{\partial}{\partial z_{2\ell}}+z_{2\ell-1}\frac{\partial}{\partial z_{2m+1}}\right) \otimes\bar\omega^{2\ell-1}\right] \in H^{0,1}_{\overline\partial}(X,T^{1,0}X)
\]
for $t\in\mathbb R$ with $|t|<\epsilon$.
It is straightforward to check that $s(t)$ satisfies the obstruction given by the Maurer-Cartan equation, so it defines a family $\{X_t\}_{|t|<\epsilon}$ of small deformations of $X=X_0$.\medskip

The holomorphic coordinates on $X_t$ are given by the solutions $\left(\zeta_\ell(t)\right)_{\ell\in\{1,\ldots,2m+1\}}$ of the initial-value PDE
\[ \left\{\begin{array}{l}
\overline\partial \zeta_\ell(t)-s(t)\zeta_\ell(t) =0 \\
\zeta_\ell(0)=z_\ell .
\end{array}
\right.
\]
It is straightforward to check that the following is a solution:
\begin{align*}
 \zeta_{2\ell-1}(t)&=z_{2\ell-1}, \quad \zeta_{2\ell}(t)=z_{2\ell}+\sqrt{-1}t\bar z_{2\ell-1}, \qquad \text{ for } \ell\in\{1,\ldots,2m\},\\
 \zeta_{2m+1}&= z_{2m+1}+\sum_{\ell=1}^{m}\sqrt{-1}t|z_{2\ell-1}|^2 . 
\end{align*}

We can then use the following coframe of invariant $(1,0)$-form on $X_t$:
\begin{align*}
\varphi^\ell(t)&=d\zeta_\ell(t) \qquad \text{ for } \ell\in\{1,\ldots,2m\},\\
\varphi^{2m+1}(t) &= d\zeta_{2m+1}(t)-\sum_{\ell=1}^{m} \left( z_{2\ell-1} d\zeta_{2\ell}(t)+\sqrt{-1}t \bar z_{2\ell-1}d\zeta_{2\ell-1}(t)\right) . 
\end{align*}
The structure equations in this coframe are
\begin{align*}
 d\varphi^\ell(t)&=0 \qquad \text{ for } \ell\in\{1,\ldots,2m\},\\
d\varphi^{2m+1}(t)&= -\sum_{\ell=1}^{m} \varphi^{2\ell-1}(t) \wedge \varphi^{2\ell}(t)+\sum_{\ell=1}^{m}\sqrt{-1}t \varphi^{2\ell-1}(t)\wedge\bar\varphi^{2\ell-1}(t).
\end{align*}
For $t=1$, we recover the previous examples, but the same argument there applies for any $t\neq 0$.
\end{rem}
\subsection{Fr\"olicher differentials in dimension $3$}\label{sec: h15}
In \cite{ceballos_invariant_2016}, the Fr\"olicher spectral sequence of left-invariant structures on nilmanifolds of real dimension $6$ is studied. Only three nilpotent Lie algebras admit left-invariant complex structures with $E_2\neq E_3$, namely $\fh_{13}, \fh_{14}$ and $\fh_{15}$. The last example has a particulary rich behaviour: The left-invariant complex structures are parametrized by a triple $(\rho,B,C)\in\{0,1\}\times\R_{>0}\times\C$ (subject to some conditions), and depending on the parameter values, the behaviour of the Fr\"olicher spectral sequence changes. In fact, inspection of the proof of \cite[Thm. 4.1]{ceballos_invariant_2016} yields the following equalities in $\cU_3/\RB_3^{form}$:

\begin{align*}
\rho=c=0, B=1: &&[X_{(\rho,c,B)}]&\equiv \begin{gathered}\includegraphics[scale=1]{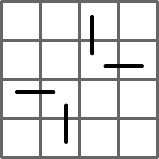}\end{gathered}+2\cdot\begin{gathered}\includegraphics[scale=1]{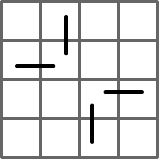}\end{gathered}\\
\rho=1, c=0, B\neq 1: && [X_{(\rho,c,B)}]&\equiv\begin{gathered}\includegraphics[scale=1]{S1102_3.pdf}\end{gathered}  + 2\cdot\begin{gathered}\includegraphics[scale=1]{S1111_3.pdf}\end{gathered}+\delta_{0}^B\cdot\begin{gathered}\includegraphics[scale=1]{S1101_3.pdf}\end{gathered}\\
\rho=1,|B-1|\neq c\neq 0:&&[X_{(\rho,c,B)}]&\equiv \begin{gathered}\includegraphics[scale=1]{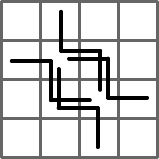}\end{gathered}\\
\rho=0, |B|\neq c\neq 0:&&[X_{(\rho,c,B)}]&\equiv \begin{gathered}\includegraphics[scale=1]{S1101_3.pdf}\end{gathered} + \begin{gathered}\includegraphics[scale=1]{S1202_3.pdf}\end{gathered} 
\end{align*}

Since we have already seen that $\RB_3=\RB_3^{form}$, we obtain that
\[
Sym_{S_{1,2}^{0,2}}(3), Sym_{S_{1,1}^{0,1}}(3), Sym_{S_{1,1}^{0,2}}(3)\in \cU_3,
\]
leaving open only the question whether or not $Sym_{S_{1,2}^{0,1}}(3)\in\cU_3$. An inspection of the other two possibilities of nilmanifolds with nonzero $E_2$-differentials yields:
\begin{prop}\label{prop: no d201 on nilmfds}
For all left-invariant structures on $6$-dimensional nilmanifolds, the differential $d_2^{0,1}:E_2^{0,1}\to E_2^{2,0}$ vanishes. In particular, it is not possible to realize $Sym_{S_{1,2}^{0,1}}(3)$ using complex nilmanifolds.
\end{prop}

\section{Bimeromorphic invariants, reprise}

We prove Theorem \ref{corintro: bimeromorphic coh. invs}. Let $\cB^{\RB}\subseteq \RB_\ast$ and $\cB^{\cU}\subseteq\cU_\ast$ be the ideals generated by differences of pairs of bimeromorphic manifolds.
\begin{lem}\label{lem: ideal bimeromorphic diffs}
The ideals $\cB^{\RB}\subseteq \RB_\ast$ and $\cB^{\cU}\subseteq\cU_\ast$ are principal ideals generated by \[
C=Sym_2^{1,1}(2)=\img{S211_2}.
\]
\end{lem}
\begin{proof}
If $Z\subseteq X$ is a submanifold of codimension $r\geq 2$ and $\tilde{X}$ the blowup of $X$ at $Z$, then $[\tilde{X}]-[X]=\sum_{i=1}^{r-2}C^i\cdot[Z]z^{r-1-i}\in\cU_\ast$ by \cite{stelzig_double_2019}. In particular, taking a point blowup shows that $C$ is contained in the respective ideals. On the other hand, every bimeromorphic map is the composition of blow-ups and blowdowns by the weak factorisation theorem \cite{abramovich_torification_2002}, \cite{wlodarczyk_toroidal_2003}, hence the other inclusion holds.
\end{proof}
It is open whether the inclusions $\RB_\ast\subseteq RB_\ast^{form}$ and $\cU_\ast\subseteq\cU_\ast^{form}$ of the refined de Rham ring and the universal ring of cohomological invariants into their formal counterparts from Definition \ref{def. formal de Rham ring} and Definition \ref{def: univ form} are equalities in any degree. Nevertheless, in view of the previous Lemma we will now study the principal ideals generated by $C$ in these potentially bigger rings. The equalities in low degrees will allow us to bring this to bear on the original, geometrically defined rings.
\begin{lem}\label{lem: bases for formal rings mod C}The following holds for the graded rings $\RB_\ast^{form}$ and $\cU_\ast^{form}$:
\begin{enumerate}
\item The polynomials $Sym_n^{n,0}(n)$, $Sym_2^{2,2}$ and
\begin{align*}
Sym_k^{p,0}(n)&\text{ for } 0\leq p\leq k\leq n-1,\\
Sym_k^{k, p}(n)&\text{ for } 0\leq p\leq k\leq n-1,\\
Sym_n^{n-1,p}(n)&\text{ for } 2\leq p\leq n-1,
\end{align*}
project to a basis for the degree $n$ part of $\RB_\ast^{form}/\RB_\ast^{form}\cdot C$.
\item The degree $n$ part of $\cU_\ast^{form}/\cU_\ast^{form}\cdot C$ has a basis consisting of the $Sym_Z(n)\in\cU_n$ where $Z$ is a zigzag satisfying at least one of the following conditions:
\begin{enumerate}
\item the support of $Z$ contains a point $(p,q)$ with $p$ or $q$ equal to $0$ or $n$,
\item the length of $Z$ is at least two and $(1,1)\in Z$,
\item the length of $Z$ is at least two and $(n-1,1)\in Z$,
\item $[Z]=S_{1,1}^{1,2}$ if $n=4$.
\end{enumerate} 
\end{enumerate}
\end{lem}
Part $1$ is a special case of part $2$ which is readily seen by the diagrammatic description and the restrictions given in Definition \ref{def: univ form} and we omit a formal proof. For example, a basis for the degree $3$ part of $\cU_\ast^{form}/\cU_\ast^{form}\cdot C$ is given by all $Sym_Z(3)\in\cU_3^{form}$ except the following:

\[
\img{S211_3}\quad\text{ and }\quad\img{S321_3}
\]

\begin{cor}\label{cor: bimeromorphic invariants on formal rings}
Let $R$ be a ring. Any map from $(\RB_\ast^{form}/(C))_n\otimes R$ (resp. $(\cU_\ast^{form}/(C))_n\otimes R$) to $R$ can be written as an $R$-linear combination of the refined Betti numbers (resp. multiplicities of zigzags) that are the coefficients of the basis elements in Lemma \ref{lem: bases for formal rings mod C}. 
\end{cor}
Let $K$ be a field. As before, we say a $K$-linear combination of cohomological invariants of compact complex $n$-folds is a map $\cU_n\to K$. We say such a combination is a universal bimeromorphism invariant if it vanishes on $\cB^{\cU}_n$. Similarly, a $K$-linear combination of refined Betti numbers of compact complex $n$-folds is a map $\RB_n\to K$ and we say such a combination is a universal bimeromorphism invariant if it vanishes on $\cB^{\RB}_n$.
\begin{cor}[= Theorem \ref{corintro: bimeromorphic coh. invs}]
	Let $K$ be a field. Any $K$-linear combination of cohomological invariants (resp. refined Betti numbers) which is a universal bimeromorphism invariant can be written as a $K$-linear combination of multiplicities of zigzags (resp. refined Betti numbers) that are the coefficients of the basis elements in Lemma \ref{lem: bases for formal rings mod C}. 
\end{cor}
\begin{proof}[Proof of Theorem \ref{corintro: bimeromorphic coh. invs}]
By Lemma \ref{lem: ideal bimeromorphic diffs}, there is a natural map $\eta:\cU/\cU_\ast\cdot\cB^{\cU}\to\cU_\ast^{form}/\cU_\ast^{form}\cdot C$. The map $\eta$ is injective in degree $n$ if and only if $\cU^{form}_{n-2}\cdot C\cap\cU_n=\cU_{n-2}\cdot C$. This holds for $n\leq 4$ since then $\cU_{n-2}^{form}=\cU_{n-2}$. Now let $K$ be some field and consider a $K$-linear combination of cohomological invariants of compact complex $n$-folds which is a universal bimeromorphism invariant, i.e. a map $\varphi:(\cU_\ast/ \cU^{\RB}\otimes K)_n\to K $. Since $\eta_K=\eta\otimes\Id$ is an injective map of vector spaces, we can find a $\varphi':(\cU_\ast^{form}/\cU_\ast^{form}\cdot C)_n\to K$ such that $\varphi=\varphi'\circ \eta_K$ and apply Corollary \ref{cor: bimeromorphic invariants on formal rings}. The case of $\RB_\ast$ is analogous, using that $\RB_{\leq 3}=\RB_{\leq 3}^{form}$.
\end{proof}
\section{Summary of some open problems}\label{sec: constr. probs}

Solving the following problem would show $\RB_\ast=\RB_\ast^{form}$:
\begin{prob}\label{prob: hahs}
For even $n\geq 4$, construct an $n$-dimensional compact complex manifold $X_n$ with $b_{n-1}^{n-1,n-1}(X_n)=1$, i.e. supporting a nonzero $n-1$ de Rham class, unique up to scalar, which can be represented by both a holomorphic and an antiholomorphic form.
\end{prob}
Ideally, such examples would also allow to show the slightly stronger statement
\[
Sym_{S_{n-1}^{n-1,n-1}}(n)=\img{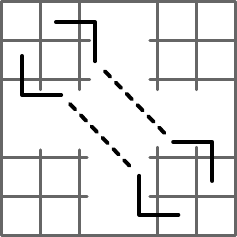}\overset{!}{\in}\cU_n.
\]
Apart from this and maybe some sporadic generators in small degrees, the main problem in showing a hypothetical equality $\cU_\ast=\cU_\ast^{form}$ should consist in finding differentials in the Fr\"olicher spectral sequence in minimal dimension and extremal degree:
\begin{prob}\label{prob: extremal FSS diffs}
For every $n\geq 3$, construct a $n$-dimensional compact complex manifold $X_n$ with nonvanishing differentials on page $E_{n-1}$ starting in degree $(0,n-1)$ or $(0,n-2)$. More precisely, where for readability we omit zigzags determined via the real structure, show that the following lie in $\cU_n$:
\[
Sym_{S_{1,n-1}^{0,n-1}}(n)=\img{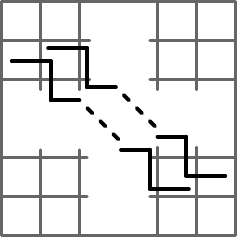} \text{ and } Sym_{S_{1,n-1}^{0,n-2}}(n)=\img{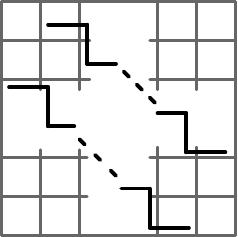}
\]
For $n=3$, only the second one is missing (compare the statements of Theorem \ref{corintro: no unexp. rels coh. invs. ref. betttis} and Proposition \ref{prop: universal ring in small degrees}).
\end{prob}

A solution to this problem would in particular solve the following problem with $r(n)=n-1$ for $n\geq 3$:

\begin{prob}\label{prob: FSS maximal nondegeneracy}
Determine for every dimension $n$, the maximal number $r=r(n)$ such that the Fr\"olicher spectral sequence of an $n$-manifold may have a nonzero differential at page $r$.
\end{prob}

It is believed that it is possible to solve problem \ref{prob: FSS maximal nondegeneracy} using nilmanifolds with possibly high nilpotency step, see \cite{bigalke_erratum_2014}. By the examples in \cite{ceballos_invariant_2016} used in Section \ref{sec: h15} one has indeed $r(3)=2$. On the other hand, by Proposition \ref{prop: no d201 on nilmfds} it is already in dimension $3$ not possible to give a positive answer to the stronger problem \ref{prob: extremal FSS diffs} using nilmanifolds only. It would be interesting to see whether Problem \ref{prob: hahs} can be solved for four-dimensional nilmanifolds.\medskip

Even stronger than the question $\cU_\ast=\cU^{form}_\ast$, one may ask what universal linear relations between cohomological invariants and Chern numbers exist.

\begin{conj} Modulo the relations involving only cohomological invariants or only Chern numbers, the only $\Z$-linear relations or congruences between cohomological invariants and Chern numbers of compact complex manifolds are the ones induced by Hirzebruch-Riemann-Roch.
\end{conj}

\begin{rem}
If one finds a generating set for $\cU_\ast$ consisting of $\CP^1,\CP^2$ and otherwise only manifolds with vanishing Chern classes (e.g. complex nil- or solvmanifolds), this would follow with the same proof as Theorem \ref{thm: HDRC ring diagonal}. 
\end{rem}

For completeness, note that to address the integral version of the Hirzebruch question for Hodge and Betti numbers, in Section \ref{sec: rat. Hirzebruch} it remained to solve the following:
\begin{prob}\label{prob: 4-folds for integral statement}
Find $\Z$-linear combinations $Z=\sum a_i\cdot (X_i-Y_i)$, where $X_i$ and $Y_i$ are homeomorphic, or even orientation preservingly diffeomorphic, complex $4$-folds, such that each of the following can be realized as $h(Z)$:
\begin{align*}
CT&=(xy^2-x^2y+x^3y^2-x^2y^3)z^4\\
C\tilde{d}&=(x^2y+2x^2y^2+x^2y^3)z^4\\
D\tilde{d}&=(x^2+2x^2y+2x^2y^2+2x^2y^3+x^2y^4)z^4.
\end{align*} 
\end{prob}

Finally, it seems natural to ask the following generalization of Hirzebruch's question:

\begin{qu}
Which linear combinations or congruences of cohomological invariants and Chern numbers are topological invariants of compact complex manifolds?
\end{qu}

\bibliographystyle{acm}

Jonas Stelzig, Mathematisches Institut der Ludwig-Maximilians-Universit\"at M\"unchen, Theresienstraße 39, 80333 M\"unchen. jonas.stelzig@math.lmu.de.
\end{document}